\documentclass{article}
\usepackage[utf8]{inputenc}
\usepackage[usenames,dvipsnames]{xcolor}
\pdfoutput=1
\usepackage[english]{babel}
\usepackage{amsthm,amsmath,amssymb,soul}
\usepackage[all]{xy}
\usepackage[normalem]{ulem}
\usepackage{xy,mathrsfs,pstricks,verbatim,numprint}
\usepackage{color,xcolor,transparent,subfigure} \usepackage{tikz,ltablex,booktabs,tikz-cd}
\usepackage{tabstackengine}
\usetikzlibrary{arrows,decorations.pathmorphing}
\usetikzlibrary{positioning,arrows.meta,decorations.text,calc}
\usetikzlibrary{shapes.multipart}

\newdir{ >}{{}*!/-5pt/\dir{>}}
\usepackage[a4paper]{geometry}
\usepackage[utf8]{inputenc}

\newcommand{\imono}[1]{ \;\xymatrix{  \ar@{>->}^{#1}[r] &  \\} }
\newcommand{\iepi}[1]{ \;\xymatrix{  \ar@{->>}^{#1}[r] &  \\} }
\newcommand{\mono}{ \;\xymatrix{  \ar@{>->}[r] &  \\} }
\newcommand{\epi}{ \xymatrix{   \ar@{->>}[r] &  \\} }

\def\rep{\mbox{rep}\,}

\def\add{\mbox{add}}

\newcommand{\bR}{\mathbb{R}}
\newcommand{\A}{\mathcal{A}}
\newcommand{\B}{\mathcal{B}}
\newcommand{\E}{\mathcal{E}}

\newcommand{\cF}{\mathcal{F}}

\renewcommand{\P}{\mathcal{P}}

\newcommand{\T}{\mathcal{T}}

\newcommand{\Hom}{Hom}
\newcommand{\bsm}{\left[ \begin{smallmatrix}}
\newcommand{\esm}{\end{smallmatrix} \right]}

\theoremstyle{definition} 
\newtheorem{theorem}{Theorem}[section]
\newtheorem{proposition}[theorem]{Proposition}
\newtheorem{lemma}[theorem]{Lemma}
\newtheorem{cor}[theorem]{Corollary}
\newtheorem{remark}[theorem]{Remark}
\newtheorem{example}[theorem]{Example}
\newtheorem{definition}[theorem]{Definition}

\theoremstyle{definition}

\newcommand{\cS}{\mathcal{S}}
\newcommand{\mN}{\mathbb{N}}
\def\End{\mbox{End}}
\def\Hom{\mbox{Hom}}
\def\Ext{\mbox{Ext}}
\def\Ban{\mbox{\bf Ban}}

\def\rep{\mbox{rep}\,}
\def\mod{\mbox{mod}\,}

\def\Im{\mbox{Im}\, }
\def\Coim{\mbox{Coim}\,}
\def\Ker{\mbox{Ker}\,}
\def\Coker{\mbox{Coker}\,}
\def\radical{\mbox{rad}_{\E}}
\def\inter{\mbox{Int}}
\def\Max{\mbox{Max}}
\def\Min{\mbox{Min}}
\def\coker{\mbox{coker}\,}
\def\Sup{\mbox{Sup}}
\def\Sub{\mbox{Sub}}
\def\Sum{\mbox{Sum}}

\title{Intersections, sums, and the Jordan-H\"older property for exact categories}
\author{Thomas Br\"ustle, Souheila Hassoun and Aran Tattar}
\begin{document}
\date{}
\maketitle

\bigskip
\abstract{We investigate how the concepts of intersection and sums of subobjects carry to exact categories. We obtain a new characterisation of quasi-abelian categories in terms of admitting admissible intersections in the sense of \cite{HR}. There are also many alternative characterisations of abelian categories as those that additionally admit admissible sums and in terms of properties of admissible morphisms. We then define  a generalised notion of intersection and sum which every exact category admits. Using these new notions, we define and study classes of exact categories that satisfy the Jordan-H\"older property for exact categories, namely the Diamond exact categories and Artin-Wedderburn exact categories. By explicitly describing all exact structures on $\mathcal{A}= \mbox{rep}\, \Lambda$ for a Nakayama algebra $\Lambda$ we characterise all Artin-Wedderburn exact structures on $\mathcal{A}$ and show that these are precisely the exact structures with the Jordan-H\"{o}lder property. 
}


\section{Introduction}
In a classical theorem in group theory, Camille Jordan stated in 1869 that any two composition series of the same finite group have the same number of quotients. Later, in 1889, Otto H\"older reinforced this result by proving the theorem known as the Jordan-H\"older-Schreier theorem, which states that any two composition series of a given group are equivalent, that is, they have the same  length and the same factors, up to permutation and isomorphism. 
This  theorem has been generalised to many other contexts, such as operator groups, modules over rings or general abelian categories. Most proofs use the concept of intersection and sum, which is readily available for groups, modules or objects in an abelian category.

In a general categorical setup, the intersection is defined as pullback of two monomorphisms, if it exists. However, in order to define a sensible cohomology theory, one needs to restrict the notion of subobjects to {\em admissible monomorphisms}, which allow to form kernel-cokernel pairs. In the context of functional analysis, for instance, this leads to the study of closed subspaces, giving rise to the structure of a quasi-abelian category. More generally, the setup is that of Quillen exact categories \cite{Qu} which generalises abelian categories. In this generality, one requires not only that the intersection of admissible subobjects exists, but it needs to be an admissible subject itself. Central object of study in this paper is therefore the notion of admissible intersections and sums in an exact category.
\bigskip

The notion of exact categories has been recently the center of many works, see e.g.
\cite{J,E16,E17,E18,E19,E20}. 
The exact structure for Delta-filtered modules has been studied in \cite{BrHi}, and more recently in \cite{KKO}. They satisfy the Jordan-H\"older property, which is also shown in \cite{Sa} in the context of  stratifying systems in exact categories.
 Given an exact category, \cite{BeGr} and \cite{FG} study  its associated Hall algebra, and \cite{VW} the graded Lie algebra. 
 Unicity of filtrations for exact categories is also studied in \cite{Ch} related to the Harder-Narasimhan property for exact categories. And the Jordan-H\"older property in the context of semilattices is studied in  \cite{P}.

Choosing a Quillen exact structure allows to define various cohomology theories for locally compact abelian groups, Banach spaces, or other categories studied in functional analysis. Other areas where exact structures appear naturally are Happel's construction of triangulated categories from Frobenius categories, or extension-closed subcategories of abelian categories. 
The set of exact structures on a fixed additive category forms a lattice $(Ex(\A), \subset)$ as shown in \cite{BHLR}. This lattice is studied recently by the first two authors in \cite{BBH}, and also by Fang and Gorsky in \cite{FG}.
Note also that the exact structures are classified by Enomoto in \cite{E16} using Wakamatsu tilting, and in \cite{E17} where he gives a classification of all exact structures on a given idempotent complete additive category.\\

We give now a more detailed description of the main concepts and results in this paper.
\begin{definition}[Definition \ref{composition series}]
Let $(\A, \E)$ be an exact category. A finite $\mathcal{E}-$composition series for an object $X$ of $\mathcal{A}$ is a sequence 
\[ \begin{tikzcd}[sep=large] 0= X_{0} \arrow[r, tail, "i_0"] &  X_1 \arrow[r, tail, "i_1"] & \dots \arrow[r, tail, "i_{n-2}"] &  X_{n-1} \arrow[r, tail, "i_{n-1}"] & X_n = X  \end{tikzcd}\] 
where all $i_l$ are \emph{proper admissible monics} with $\E-$simple cokernel. 
We say an exact category $(\A,\E)$ has the {\em ($\E-$)Jordan-H\"older property} or is a {\em Jordan-H\"older exact category}
if any two finite $\E$-composition series of $X$ are equivalent, that is, they have the same length and the same composition factors, up to permutation and isomorphism. 
\end{definition} 

This is an interesting problem since the Jordan-H\"older property does not hold in general for any exact category, see \cite[Example 6.9]{BHLR}, \cite{E19} and Examples \ref{even dimensional spaces} and \ref{Counterexample1} for counter-examples. This problem is also studied by Enomoto in \cite{E19},  using the Grothendieck monoid which is a lesser-known invariant of exact categories defined by the same universal property as the Grothendieck group. He shows that the relative Jordan-H\"older property holds if and only if the Grothendieck monoid of the exact category is free. Note that the same author considered the Grothendieck group for exact categories in \cite{E18}.
In this work we fix  an additive category $\A$ and study for which exact structures $\E\in Ex(\A)$ the relative $\E-$Jordan-H\"older property holds.

\bigskip
We also establish a generalisation of the Fourth Isomorphism Theorem for modules, which will be a useful tool throughout our work.
\begin{proposition} [Proposition \ref{4 iso}]({\bf{The fourth $\E$-isomorphism theorem}}) 
Let $(\mathcal{A}, \E)$ be an exact category and let
\[ \begin{tikzcd} X' \arrow[r, tail] & X \arrow[r, two heads] & X / X' \end{tikzcd} \] be a short exact sequence in $\E$. Then there is an isomorphism of posets 
\begin{align*}
    \{M \in \mathcal{A} \mid X' \rightarrowtail M \rightarrowtail X \} & \longleftrightarrow \{ N \in \mathcal{A} \mid N \rightarrowtail X / X' \}   \\ M & \longmapsto M / X'.
\end{align*} 
\end{proposition}

In \cite{Bau}, Baumslag gives a short proof of the Jordan-H\"{o}lder theorem for groups, by intersecting the terms of one subnormal series with those in another series. Motivated by these ideas, we generalise the abelian notions of intersection and sum to exact categories. We do this in two ways. Firstly, in Section 4, by considering intersections as pullbacks and sums as pushouts of intersections - as is the case in the abelian setting,  see \cite[Section 5]{GR} and \cite[Definition 2.6]{pop} -  we recall \emph{AI-categories} (Admissible Intersection) and \emph{AIS-categories} (Admissible Intersection and Sum) from \cite{HR}. The AI-categories are pre-abelian exact categories where admissible monics are stable under pullback along admissible monics and all such pullbacks exist (see Definition \ref{quasi-nice})
\begin{equation} \label{introdiagram}
   \begin{tikzcd}
{A} \arrow[r, tail, "i"] \arrow[d, "f"', tail] & {B} \arrow[d, "g", tail] \\
{C} \arrow[r, tail, "j"']           & {D.} \arrow[ul, phantom, "\lrcorner" very near start]
\end{tikzcd} 
\end{equation}

In a previous version of \cite{HR} the term \emph{quasi-nice} exact categories was also used. We prove in Theorem \ref{quasi-nice is qa} that the AI-categories are necessarily quasi-abelian with the maximal exact structure $\E_{max}$ in the lattice $(Ex(\A), \subseteq)$. It has been proved recently by the second author, Shah and Wegner, in \cite[Theorem 6.1]{HSW}, that the converse is also true. Hence, we have a new characterisation of quasi-abelian categories:
\begin{theorem}[Theorem \ref{HSW 6.1}]{\bf{(Br\"ustle, Hassoun, Shah, Tattar, Wegner)}}
A category $(\mathcal{A}, \E_{max})$ is quasi-abelian if and only if it is an AI-category.
\end{theorem} 

The AIS-categories are the AI-categories that satisfy the additional property that the unique induced morphism $u$ in the pushout coming from the pullback diagram (\ref{introdiagram}) is an admissible monic (see Definition \ref{AIS}): 
\[
\begin{tikzcd}
{A} \arrow[r, tail, "i"] \arrow[d, "f"', tail] & {B} \arrow[d, "l", tail]  \arrow[ddr, "g", tail, bend left]& \\
{C} \arrow[r, tail, "k"']   \arrow[drr, tail, "j"', bend right]        & {E} \arrow[ul, phantom, "\ulcorner" near end]  \arrow[dr, tail, "u"] \\ & & D       
\end{tikzcd}
\]
It turns out that the AIS-categories are precisely the abelian categories endowed with the maximal exact structure. This, along with our study of the behaviour of admissible morphisms under composition and sum in Section \ref{admissibles and abelian}, allows us to give these following alternative characterisations of abelian categories:
\begin{theorem}[Theorems \ref{abelian 1} \& \ref{abelian 2}]
Let $(\A, \E)$ be an exact category. Then the following are equivalent:
 \begin{enumerate}
     \item[a)]
     $\A$ is an abelian category and $\E=\E_{max}$,
     \item[b)] $(\A, \E)$ is an AIS-category,
     \item[c)] $\Hom(\A)=\Hom^{ad}(\A)$,
     \item[d)] $\Hom^{ad}(\A)$ is closed under composition,
     \item[e)] $\Hom^{ad}(\A)$ is closed under addition,
 \end{enumerate} where $\Hom^{ad}(\A)$ denotes the addmissible morphisms in $\A$ (see Definition \ref{ad mor}).
\end{theorem}

As we observe in  Examples \ref{even dimensional spaces} and \ref{example A2 again}, the pullback and pushout notions of unique intersection and sum do not necessarily apply for general exact categories- even if the exact category has the Jordan H\"older property. This leads us to define, in Definition \ref{general intersection and sum}, a general notion of admissible intersection and sum that works for all exact categories. For two admissible subobjects $(A, f)$ and $(B, g)$ of $X$, their intersection, $\inter_{X}(A, B)$, is the set of all their maximal common proper admissible subobjects. Dually, their sum,  $\Sum_{X}(A, B)$ is the set of all their minimal common proper admissible superobjects that are subobjects of $X$. Using this, we study exact categories satisfying the Diamond axiom.
\begin{definition}[Definition \ref{Diamond}]({\bf{Diamond Axiom}}) 
Let $X \in \A$ and let $(A,f)$ and $(B,g)$ be two distinct maximal $\E-$subobjects of $X$, that is, the cokernels $X/A$ and $X/B$ are $\E-$simple. 
We say that $(A,f)$ and $(B,g)$ satisfy the {\em diamond axiom} if for every $Y \in \inter_X(A,B)$ we have that $A/Y$ and $B/Y$ are both $\E-$simple, and are isomorphic to the $\E-$simple cokernels of $X$
and the elements of the sets
 $\{X/A, A/Y \} \{ X/B, B/Y \}$ are equal up to permuation and isomorphism

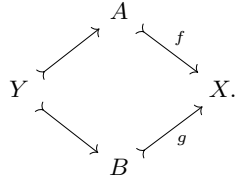
\[ \begin{tikzcd} & A \arrow[dr, tail, "f"] & \\  Y \arrow[dr, tail ] \arrow[ur, tail] &  & X. \\ & B \arrow[ur, tail, "g"'] &   \end{tikzcd}
\]
\end{definition} 
These categories generalise the abelian categories as we note in Remark \ref{ab is diamond}, and satisfy the relative Jordan-H\"older property:
\begin{theorem}[Theorem \ref{JH}] Every diamond exact category is a Jordan-H\"older exact category. 
\end{theorem}

\bigskip
Later, in Section 6, we use the new
the notion of generalised intersection to define an analog of the Jacobson radical for exact categories, the \emph{$\E-$Jacobson radical}, $\radical(X)$, as the generalised intersection of all maximal $\E$-subobjects of $X$  and also introduce the notion of \emph{$\E-$semisimple} objects (see Definitions \ref{radical} and \ref{esemisimple}). We show some basic properties of the $\E-$Jacobson radical motivated by the properties of the classical Jacobson radical. We then use this to introduce the \emph{$\E-$Artin-Wedderburn categories}, which are exact categories where an analog of the classical Artin-Wedderburn theorem holds:
\begin{definition}[Definition \ref{A-W}]
An exact category $(\A, \E)$ is called \textit{Artin-Wedderburn} if for any object $X \in A$ the following properties are equivalent:
\begin{enumerate}
    \item[(AW1)] Every sequence in $\E$ of the form $A \rightarrowtail X \twoheadrightarrow X/ A$ splits,
    \item[(AW2)] $X$ is $\E$-semisimple,
    \item[(AW3)] $\radical(X):=  \inter_X\{(Y,f) \in \mathcal{S}_X \mid \Coker f \mbox{ is } \E-\mbox{simple}\}= \{0\}$.
\end{enumerate} 
Here $\mathcal{S}_X$ is the poset of all \emph{proper} $\E-$subobjects of $X$ (see Definition \ref{poset of subobjects}).
We call in this case $\E$ an \textit{Artin-Wedderburn} exact structure on $\A$.
\end{definition}
We give examples of such categories and prove in Lemma \ref{Emin is W-A}, that every additive category with the minimal exact structure $\E_{min}$ in the lattice $(Ex(\A), \subseteq)$; the split exact structure, is an $\E-$Artin-Wedderburn category.
Then, by showing that certain $\E-$Artin-Wedderburn categories satisfy the Diamond axiom, we obtain the following result:
\begin{theorem}[Theorem \ref{thm:AW->JH}]
Let $(\A, \E)$ be a Krull-Schmidt $\E$-Artin-Wedderburn category. Then $(\A, \E)$ is a Jordan-H\"{o}lder exact category.
\end{theorem}
We then give for any Nakayama algebra, $\Lambda$,  an explicit description of all exact structures on $\rep \Lambda$ in Theorem \ref{EB description} and use this to characterise all Artin-Wedderburn exact structures on $\rep \Lambda$ in Theorem \ref{EAW example}. It turns out these they are exactly the Jordan-H\"older exact structures on $\rep {\Lambda}$:
\begin{theorem}[Theorem \ref{Nakayama}]
Let $\Lambda$ be a Nakayama algebra, and denote $\A = \mod \Lambda$, the category of finitely generated left $\Lambda-$modules.
Then an exact category $(\A,\E)$ is $\E-$Artin-Wedderburn precisely when it is Jordan-H\"older.
\end{theorem}
Once satisfied, the $\E-$Jordan-H\"older property allows  to define \emph{the $\E-$Jordan-H\"older length function} (compare also \cite[4.1]{E19}):
\begin{definition}[Definition \ref{length}] The $\E-$Jordan-H\"older length $l_{\mathcal{E}}(X)$ of 
an object $X$ in $\A$
is the length of an $\E-$composition series of $X$. That is $l_{\mathcal{E}}(X)=n$ if and only if there exists an $\E-$composition series 
\[
\begin{tikzcd}0=X_0 \arrow[r, tail] & X_1 \arrow[r, tail] & \dots \arrow[r, tail]            & X_{n-1} \arrow[r, tail] & X_n=X. \end{tikzcd}\] 
We say in this case that $X$ is $\E-$finite.
\end{definition}
This $\E-$Jordan-H\"older length function has good properties that improves the general length defined and studied on any exact category in \cite[Definition 6.1, Theorem 6.6]{BHLR} in which there is only an inequality. 
\begin{proposition}[Corollary \ref{s.e.s length}] Let $X \; \rightarrowtail Z \twoheadrightarrow$ Y be an admissible short exact sequence of finite length objects. Then

$$l_{\mathcal{E}}(Z) = l_{\mathcal{E}}(X) + l_{\mathcal{E}}(Y).$$
\end{proposition}
Moreover this length function satisfies also important properties as:
\begin{proposition}[Proposition \ref{hopkinslevitzki}] An object $X$ of $(\A, \E)$ is $\E-$Artinian and $\E-$Noetherian if and only if it has an $\E-$finite length.
\end{proposition}
Finally, the $\E-$Jordan H\"older length function can only decrease under reduction of exact structures:

\begin{proposition} [Proposition \ref{lengthreduction}]
 If $\E$ and $\E'$ are exact structures on $\A$ such that $\E' \subseteq \E$, then $l_{\E'}(X) \leq l_{\E}(X)$ for all objects $X$ in $\A$.
\end{proposition}


\paragraph{Acknowledgements.}
The authors would like to thank Charles Paquette for his thoughtful comments on the PhD thesis of the second author. The authors would also like to thank Amit Shah, Sven-Ake Wegner and Sunny Roy for all the discussions contributing to this version of the work. The authors  also thank Matthias K\"unzer and Haruhisa Enomoto for their useful comments. The authors also thank the anonymous referee for their careful reading and helpful comments which have improved the manuscript. The third author thanks their coauthors for their hospitality during visits to Université de Sherbrooke and Sibylle Schroll for helpful discussions. \\
The first two authors were supported by Bishop's University, Université de Sherbrooke and NSERC of Canada. The second author is supported by the scholarship "thésards étoiles" of the ISM. The third author was supported by the EPSRC.


\section{Background}


In this section we recall from \cite{GR,Bu} the definition of a Quillen exact structure along with the definitions of various types of additive categories and other important concepts that form the backdrop to our work.
\subsection{Exact structures on additive categories}
\begin{definition}
An \emph{additive category} $\A$ is a preadditive category (all its hom-sets are abelian groups and composition of morphisms is bilinear) admitting all finite biproducts.
\end{definition}
\begin{definition}\label{exact structure}Let $\mathcal{A}$ be an additive category. A kernel-cokernel pair $(i, d)$ in $\mathcal{A}$ is a pair of composable morphims such that $i$ is kernel of $d$ and $d$ is cokernel of $i$.
If a class $\mathcal{E}$ of kernel-cokernel pairs on $\mathcal{A}$ is fixed, an {\em admissible monic} is a morphism $i$ for which there exist a morphism $d$ such that $(i,d) \in \mathcal{E}$. An {\em admissible epic} is defined dually. Note that admissible monics and admissible epics are referred to as inflation and deflation in \cite{GR}, respectively. 
We depict an admissible monic by  $\rightarrowtail$
and an admissible epic by $\twoheadrightarrow$.
An {\em exact structure} $\mathcal{E}$ on $\A$ is a class of kernel-cokernel pairs $(i, d)$ in $\A$ which is closed under isomorphisms and satisfies the following axioms:
\begin{enumerate}
\item[(A0)] For all objets  $A \in Obj\mathcal{A}$ the identity $1_A$ is an admissible monic,
\item[{(A0)$^{op}$}] For all objets  $A \in Obj\mathcal{A}$ the identity $1_A$ is an admissible epic,
\item[(A1)] The class of admissible monics is closed under composition,
\item[{(A1)}$^{op}$] The class of admissible epics is closed under composition,
\item[(A2)] 
 The pushout of an admissible monic $i: A \rightarrowtail B$ along an arbitrary morphism \mbox{$f: A \to C$} exists and yields an admissible monic $j$:
\[
\begin{tikzcd}
{A} \arrow[r, tail, "i"] \arrow[d, "f"'] & {B} \arrow[d, "g"] \\
{C} \arrow[r, tail, "j"']           & {D,} \arrow[ul, phantom, "\ulcorner" near end]          
\end{tikzcd}
\]

\item[{(A2)}$^{op}$]The pullback of an admissible epic $h$ along an arbitrary morphism $g$ exists and yields an admissible epic $k$
\[
\begin{tikzcd}
{A} \arrow[r, two heads, "k"] \arrow[d, "f"'] & {B} \arrow[d, "g"] \\
{C} \arrow[r, two heads, "h"']           & {D.} \arrow[ul, phantom, "\lrcorner" very near start]          
\end{tikzcd}
\]
\end{enumerate}
An {\em exact category} is a pair $(\mathcal{A}, \mathcal{E})$ consisting of an additive category $\mathcal{A}$ and an exact structure $\mathcal{E}$ on $\mathcal{A}$. The pairs $(i,d)$ forming the class $\mathcal{E}$ are called {\em admissible short exact sequences}, {\em $\E$-sequences}, or just {\em admissible sequences.}

\end{definition}

\begin{definition}
\cite[Definition 8.1]{Bu}\label{ad mor}
A morphism $f: A \to B$ in an exact category is called admissible if it factors as $f = me$ where $m$ is an admissible monic and $e$ is
an admissible epic. Admissible morphisms will sometimes be displayed as 
\[ \begin{tikzcd} A\arrow[r, "\circ" description,"f"] & B \end{tikzcd} \]
in diagrams, and the classes of admissible arrows of $\A$ will be denoted as ${\Hom^{ad}_{\A}}(-,-)$.
\end{definition}

\begin{proposition}
\cite[Proposition 2.16]{Bu}\label{obscure axiom}
Suppose that $i: A\rightarrow B$ is a morphism in $\A$ admitting a cokernel. If there exists a morphism $j:B\rightarrow C$ such that the composition $j\circ i: A\rightarrowtail C$ is an admissible monic, then $i$ is an admissible monic.
\end{proposition}

\begin{definition}An additive category $\A$ is  \emph{pre-abelian} if it has kernels and cokernels.
\end{definition}
\begin{remark}Let $\A$ be a \emph{pre-abelian} category, then it admits pullbacks and pushouts.
\end{remark}

\begin{definition}\cite[page 524]{RW77}  A kernel $(A,f)$ is called \emph{semi-stable} if for every existing pushout square 
\[
\begin{tikzcd}
{A} \arrow[r, "f"] \arrow[d, "t"'] & {B} \arrow[d, "s_B"] \\
{C} \arrow[r, "s_C"'] & {S} \arrow[ul, phantom, "\ulcorner" near end]       
\end{tikzcd}
\]

\noindent the morphism $s_C$ is also a kernel.
We define dually a \emph{semi-stable} cokernel.
A short exact sequence $ \begin{tikzcd} A \arrow[r, tail, "i"] & B \arrow[r, "d", two heads] & C \end{tikzcd} $ is said to be \emph{stable} if
$i$ is a semi-stable kernel and $d$ is a semi-stable cokernel. We denote by $\E_{sta}$ the class of all \emph{stable} short exact sequences.
\end{definition}
\begin{remark}\cite[Theorem 3.3]{SW11}
In a pre-abelian category $\A$, the class $\E_{sta}$ forms an exact structure.
\end{remark}

\begin{definition}Let $\A$ be a pre-abelian category. A morphism $f$ is called \emph{strict} if the canonical map 
$\bar{f}$ (which is the unique morphism that exists between the image and the coimage of $f$) is an isomorphism: $Coim(f)\cong Im(f)$.\\
A short exact sequence $ \begin{tikzcd} A \arrow[r, tail, "i"] & B \arrow[r, "d", two heads] & C \end{tikzcd}$ is said \emph{strict} if
$i$ is strict or $d$ is strict. We denote by $\E_{str}$ the class of all \emph{strict} short exact sequences.\\
The class $\E_{str}$ defines the maximal exact structure $\E_{max}=\E_{str}$ in any pre-abelian category, as shown in \cite{SW11}.

\end{definition}
\begin{remark}We denote by $\E_{all}$ the \emph{class} of all short exact sequences in an additive category $\A$. We use the notation $\E_{all}$ despite the fact that the class $\E_{all}$ does not necessarily form an exact structure in an additive category $\A$.
\end{remark}
\begin{definition}\label{quasi abelian}An additive category $\A$ is \emph{quasi-abelian} if it is \emph{pre-abelian} and  $\E_{all}=\E_{sta}$.\end{definition}
It is clear that an additive category $\A$ is quasi-abelian if it is pre-abelian and every pullback of a strict epimorphism is a strict epimorphism, and every pushout of a strict monomorphism is a strict monomorphism.

\begin{definition}An additive category $\A$ is \emph{abelian} if and only if it is \emph{pre-abelian} and $\E_{all}=\E_{str}$. 
\end{definition}
Hence abelian categories with their maximal exact structure $(\A, \E_{all})$ are the pre-abelian additive categories where every morphism is admissible.

It is well known that the class of all split short exact sequences forms an exact structure on every additive category, called the  minimal exact structure $\E_{min}$. Note that certain properties of the underlying additive category $\A$  determine which exact structures can exist on $\A$. See \cite[Section 2]{BHLR} for a summary on the minimal and maximal exact structures on any additive category.
  Moreover, under some finiteness conditions, the exact structures on $\A$ are parametrized by subsets of Auslander-Reiten sequences. 
This phenomenon was observed in \cite[Theorem 5.7]{BHLR}, and is based on \cite{E18}:
\begin{theorem}\label{cube of AR-sequences}
Let $\A$ be a skeletally small, Hom-finite, idempotent complete additive category which has finitely many indecomposable objects up to isomorphism. Then every exact structure $\E$ on $\A$ is uniquely determined by the set $\mathcal B$ of Auslander-Reiten sequences that are contained in $\E$. We write in this case $\E=\E(\mathcal B).$
\end{theorem}

\subsection{The poset of $\E-$subobjects}
Now let us also recall the following useful and well known notions:
\begin{definition}\label{lattice}
A poset $P$ is called a \emph{lattice} if for every pair of elements of $P$ there exists a supremum and an infimum. In other words, there exist two binary operations $\vee$ and $\wedge:$ $P\times P \rightarrow P$ satisfying the following axioms:
\begin{enumerate}
    \item  $\vee$ is associative and commutative,
    \item  $\wedge$ is associative and commutative,

    \item  $\wedge$ and $\vee$ satisfy the following property:
    \[m \vee (m \wedge n)=m=m\wedge (m\vee n)\;\mbox{ for all } m, n \in P.\]
\end{enumerate}

\end{definition}
\begin{remark}
As a consequence of the axioms above we have the following property for lattices: 
\[m\vee m = m \;\mbox{ and }\; m\wedge m = m \; \mbox{ for all } m \in P.
\]
\end{remark}
\medskip

\begin{definition}\label{modular}
A lattice $(P, \leq, \wedge, \vee)$ is \emph{modular}
if the following property is satisfied  for all $x_1, x_2, x_3 \in P$ with $x_1 \le x_2$:
\[x_2\wedge (x_1 \vee x_3)=x_1 \vee (x_2 \wedge x_3).\]
\end{definition}
 
\begin{definition}\cite[Definition 3.1]{BHLR} 
Let $A$ and $B$ be objects of an exact category $(\mathcal{A},\E)$. If there is an admissible monic $i: A \rightarrowtail B$ we say the pair $(A,i)$ is an {\em admissible subobject} or {\em $\mathcal{E}-$subobject of $B$}. Often we will refer to the pair $(A,i)$ by the object $A$ and write $A {\subset}_{\mathcal{E}} B $.
If $i$ is not an isomorphism, we use the notation  $A {\subsetneq}_{\mathcal{E}} B $ and if, in addition, $A \not \cong 0$ we  say that $(A,i)$ is a \emph{proper} admissible subobject of $B$.
\end{definition}

\begin{definition}\cite[Definition 3.3]{BHLR}\label{simple}
A non-zero object $S $ in $(\A,\E)$ is {\em $\mathcal{E}-$simple} if $S$ admits no $\E-$sub\-objects except $0$ and $S$, that is, whenever $ A \subset_\E S$, then $A$ is the zero object or isomorphic to  $S$.
\end{definition}

\begin{remark}\label{quotient}
Let $A$ be an $\mathcal{E}-$subobject of $B$ given by the monic $i: A \rightarrowtail B$.
We denote by $B{/}^{i}A$ (or simply $B/A$ when $i$ is clear from the context) the cokernel of $i$, thus we denote the corresponding admissible sequence as
\[\begin{tikzcd} A \arrow[r, tail, "i"] & B \arrow[r,  two heads] & B / A \end{tikzcd}\]
\end{remark}

\begin{remark}\label{zero coker}
An admissible monic $i: A \rightarrowtail B$ is proper precisely when its co\-kernel is non-zero. 
In fact, by uniqueness of kernels and cokernels, the exact sequence 
$$\begin{tikzcd} B \arrow[r, tail, "1_B"] & B \arrow[r,  two heads] & 0 \end{tikzcd}$$ 
is, up to isomorphism, the only one with zero cokernel. Thus an admissible monic $i$ has $\Coker(i) = 0$ precisely when $i$ is an isomorphism. Dually, an admissible epic $d: B \rightarrowtail C$ is an isomorphism precisely when $\Ker(d) = 0$. In particular a morphism
which is at the same time an admissible monic and epic is an isomorphism.
\end{remark}

\begin{definition}\cite[Section 6.1]{BHLR}
An object $X$ of $(\A, \E)$ is \emph{$\E-$Noetherian} if any increasing sequence of $\E-$subobjects of $X$ 

\[ \begin{tikzcd} X_1 \arrow[r, tail] &  X_2 \arrow[r, tail] &  \dots \arrow[r, tail] &  X_{n-1} \arrow[r, tail] & X_n \arrow[r, tail] & X_{n+1} \dots  \end{tikzcd}\]
becomes stationary.
Dually, an object $X$ of $(\A, \E)$ is \emph{$\E-$Artinian} if any descending sequence of $\E-$subobjects of $X$

\[\begin{tikzcd} \dots X_{n+1} \arrow[r, tail] &  X_n \arrow[r, tail] & X_{n-1} \arrow[r, tail] &  \dots \arrow[r, tail] & X_2 \arrow[r, tail] & X_{1}  \end{tikzcd} \] 
becomes stationary.
 An object $X$ which is both $\E-$Noetherian and $\E-$Artinian is called \emph{$\E-$finite}. 
The exact category $(\A, \E)$ is called \emph{$\E-$Artinian} 
(respectively \emph{$\E-$Noetherian}, \emph{$\E-$finite}) if every object is $\E-$Artinian (respectively $\E-$Noetherian, $\E-$finite).
\end{definition}
Now let us recall a  definition similar to \cite[Definition 2.1]{E19}:

\begin{definition}\label{poset of subobjects}
Two $\E$-subobjects $(\begin{tikzcd} Y_i \arrow[r, tail, "f_i"] & X\end{tikzcd})$ for $i = 0,1 $ are \textit{isomorphic $\E$-subobjects of $X$} if there exists an isomorphism $\phi \in \A(Y_0, Y_1)$ such that $f_0=f_1 \circ \phi$. 
 We denote by $\P^\E_X$  the set of isomorphism classes of  $\E$-subobjects of $X$. The relation \[(Y,f)\leq (Z,g) \Longleftrightarrow \exists \begin{tikzcd} Y \arrow[dr, tail, "f"] \arrow[r, tail, "\exists h"] & Z \arrow[d, tail, "g"] \\ & X\end{tikzcd}\]
turns  $(\P^\E_X, \leq)$ into a poset. Sometimes, to avoid clutter, we will drop the superscript $\E$ and write $\mathcal{P}_X$. By $\mathcal{S}_X$ we denote the set of isomorphism classes of proper $\E$-subobjects of $X$, thus $\P^\E_X = \mathcal{S}_X \cup \{0\} \cup \{X\}$ and $\mathcal{S}_X$ inherits a poset structure from $\mathcal{P}_X$.
 \end{definition}

\begin{remark}\label{remark poset P_X} 
An $\E$-subobject $(Y,f)$ of $X$ is a maximal element of $\P^\E_X$ if and only if $\Coker f$ is $\E$-simple. For a poset $(P, \leq)$, by $\Max(P)$ we denote the maximal elements of the poset. Thus $\Max(\mathcal{S}_X)$ is the class of \emph{maximal $\E$-subobjects} of $X$. 
\end{remark}

\section{General results}

We show an $\E-$version of the fourth isomorphism theorem. We also give some results describing the behaviour of admissible morphisms, which yields a new characterisation of abelian categories in Theorem \ref{abelian 1}. 


\subsection{Admissible morphisms and abelian categories} \label{admissibles and abelian}
In this subsection we show that the admissible morphisms in an exact category behave poorly, unless we work in an abelian category with the maximal exact structure. Let us first recall the following related results:
\begin{proposition}\cite[Lemma 3.5]{HR}({\bf{The $\E$-Schur lemma}})\label{schur} 
Let $\begin{tikzcd} X\arrow[r, "\circ" description, "f"] & Y \end{tikzcd}$
be an admissible non-zero morphism. Then the following hold:
\begin{enumerate}
\item[a)]if $X$ is $\E-$simple, then $f$ is an admissible monic,
\item[b)]if $Y$ is $\E-$simple, then $f$ is an admissible epic.
\end{enumerate}
\end{proposition}

\begin{cor}\cite[Corollary 3.6]{HR}\label{ad endos} 
 Let $S$ be an $\E-$simple object, then the non-zero admissible endomorphisms $\begin{tikzcd} S\arrow[r, "\circ" description, "f"] & S \end{tikzcd}$ form the group Aut$(S)$ of automorphisms of $S$.
\end{cor}

\begin{remark}
The classical Schur lemma on abelian categories states that the endomorphism ring of a simple object is a division ring. We show in Corollary \ref{ad endos} that any non-zero admissible endomorphism of an $\E-$simple object is invertible, but it is not true in general that the set of admissible endomorphisms forms a ring.
 In fact, the composition of admissible morphisms need not be admissible, (see \cite[Remark 8.3]{Bu}), nor is it true for sums of admissible morphisms, as we discuss in this section. 
\end{remark}
 The following fact will our main tool:
\begin{lemma}\cite[Proposition 3.1]{Freyd}\label{F66}
Suppose that every morphism in $\A$ is admissible, then $\A$ is abelian and $\E = \E_{max}= \E_{all}$.
\end{lemma}

\begin{lemma}\label{composition}
Suppose that the class of admissible morphisms in $\A$ is closed under composition. Then $\A$ is abelian and $\E = \E_{max}= \E_{all}$.
\end{lemma}
\begin{proof}
We show that every morphism can be written as the composition of a section followed by a retraction. Whence the claim will follow from Lemma \ref{F66} since sections and retractions are always admissible morphisms, since the split exact structure is the minimal exact structure on any additive category. To this end, let $f: X \to Y $ be an arbitrary morphism in $\mathcal{A}$ and consider the two split short exact sequences 
\[ \begin{tikzcd} X \arrow[r, tail,"{ \left[ \begin{smallmatrix} 1 \\ 0 \end{smallmatrix} \right]}"] & X \oplus Y \arrow[r, two heads,"{\left[ \begin{smallmatrix}  0 \; 1 \end{smallmatrix}\right] }"] & Y  \\ X \arrow[r, tail,"{\left[ \begin{smallmatrix} 1 \\ f \end{smallmatrix} \right]}"] & X \oplus Y \arrow[r, two heads,"{\left[ \begin{smallmatrix}  -f \; 1 \end{smallmatrix} \right]}"] & Y. \end{tikzcd} \]
Then there is a commutative diagram 
\[ \begin{tikzcd} X \arrow[rr, "f"] \arrow[dr, tail, "{\left[ \begin{smallmatrix} 1 \\ f \end{smallmatrix} \right]}"']  & & Y \\ & X \oplus Y \arrow[ur, two heads, "{\left[ \begin{smallmatrix}  0 \; 1 \end{smallmatrix} \right]}"'] & \end{tikzcd} \] which proves the claim.
\end{proof}
\begin{lemma}\label{addition}
Suppose that the class of admissible morphisms in $\A$ is closed under addition. Then $\A$ is abelian and $\E = \E_{max} = \E_{all}$.
\end{lemma}
\begin{proof}
Let $f: X \to Y $ be a morphism in $\mathcal{A}$. Then 
\[ {\left[ \begin{smallmatrix} 0 \\ f \end{smallmatrix} \right]} = {\left[ \begin{smallmatrix} 1 \\ f \end{smallmatrix} \right]} + {\left[ \begin{smallmatrix} -1 \\ 0 \end{smallmatrix} \right]}: X \to X \oplus Y \] is the sum of two sections and is hence admissible by assumption. Let 
\[ \begin{tikzcd} X \arrow[rr, "{\left[ \begin{smallmatrix} 0 \\ f \end{smallmatrix} \right]}"] \arrow[dr, two heads, "{g}"']  & & X \oplus Y  \\ & Z \arrow[ur, tail, "{\left[ \begin{smallmatrix}  h' \\ h\; \end{smallmatrix} \right]}"'] & \end{tikzcd} \] be a factorisation of ${\left[ \begin{smallmatrix} 0 \\ f \end{smallmatrix} \right]}$ into an admissible epic followed by an admissible monic. Observe that, as $g$ is epic, $h' =0$. We claim that if ${\left[ \begin{smallmatrix} 0 \\ h \end{smallmatrix} \right]}$ is an admissible monic then so is $h$, whence $f = hg$ and is therefore admissible from which the statement will follow from Lemma \ref{F66}. Observe that $\Coker{\left[ \begin{smallmatrix} 0 \\ h \end{smallmatrix} \right]} \cong X \oplus \Coker h$ and consider the pullback of short exact sequences
\[ \begin{tikzcd}[sep=large] Z \arrow[r, tail, "\alpha"] \arrow[d, equal] & P \arrow[r, two heads, "\beta"] \arrow[d, "{\left[ \begin{smallmatrix} \gamma\; \\ \gamma' \end{smallmatrix} \right]}"] & \Coker h \arrow[d, "{\left[ \begin{smallmatrix} 0 \\ 1 \end{smallmatrix} \right]}"]
 \\ Z \arrow[r, tail, "{\left[ \begin{smallmatrix} 0 \\ h \end{smallmatrix} \right]}"'] & X \oplus Y \arrow[r, two heads, "{\left[ \begin{smallmatrix} 1 \;\quad 0  \\ 0 \; \coker h \end{smallmatrix} \right]}"'] &  X \oplus \Coker h. \arrow[ul, "\lrcorner" near start, phantom] \end{tikzcd} \] It is straightforward to verify that $P \cong Y$, $\beta = \coker h$, ${\left[ \begin{smallmatrix} \gamma \\ \gamma' \end{smallmatrix} \right]} = {\left[ \begin{smallmatrix} 0 \\ 1 \end{smallmatrix} \right]}$ and $\alpha = h$. Thus we are done.
\end{proof}
This shows that, in general, the set of admissible endomorphisms $\End^{ad}_\mathcal{A}(X)$ 
 is not a subring of $\End_\mathcal{A}(X)$ under the usual addition and composition, also that $\Hom^{ad}(X,Y)$ is not a group under the usual addition.  To finish, we summarise the results of this subsection.
 \begin{theorem}\label{abelian 1}The following conditions are equivalent:
 \begin{enumerate}
     \item[a)]
     $\A$ is an abelian category,
     \item[b)] $\Hom(\A)=\Hom^{ad}(\A)$,
     \item[{c)}] $\Hom^{ad}(\A)$ is closed under composition,
     \item[d)] $\Hom^{ad}(\A)$ is closed under addition.
 \end{enumerate}
 \end{theorem}
 \begin{proof}
 We know that in an abelian category $\A$ every morphism is admissible so $\Hom(\A)=\Hom^{ad}(\A)$ and it is closed under the composition and the addition of the category $\A$.\\
 The converse is follows from  Lemmas \ref{F66},
 \ref{composition} and
 \ref{addition}.
 \end{proof}
 \subsection{Isomorphism theorem}
We give a generalisation of the fourth isomorphism theorem for modules to exact categories: 

\begin{proposition}({\bf{The fourth $\E$-isomorphism theorem}})\label{4 iso}
Let $(\mathcal{A}, \E)$ be an exact category and let
\[ \begin{tikzcd} X' \arrow[r, tail] & X \arrow[r, two heads] & X / X' \end{tikzcd} \] be a short exact sequence in $\E$. Then there is an isomorphism of posets 
\begin{align*}
    \{M \in \mathcal{A} \mid X' \rightarrowtail M \rightarrowtail X \} & \longleftrightarrow \{ N \in \mathcal{A} \mid N \rightarrowtail X / X' \} = \P^\E_{X / X'}  \\ M & \longmapsto M / X'.
\end{align*} 
\end{proposition}
\begin{proof}Let us begin by showing that the correspondence is bijective.
First we note that the map $M \mapsto M / X'$ is well-defined by \cite[Lemma 3.5]{Bu}. Next, we define an inverse map $\phi$.  For $N \rightarrowtail X/X' $ define $\phi (N)$ to be the pullback 
\[ \begin{tikzcd} X' \arrow[d, equal] \arrow[r, tail] & \phi(N) \arrow[d, "\alpha"] \arrow[r, two heads] \arrow[dr, phantom, "\lrcorner", near end] & N \arrow[d, tail] \\ X' \arrow[r, tail]  & X \arrow[r, two heads] & X / X'. \end{tikzcd} \] We observe that by \cite[Proposition 2.15]{Bu}, $\alpha$ is an admissible monic and thus $\phi$ is a well-defined map. 
We now show that the maps are mutually inverse. For $X' \rightarrowtail M \rightarrowtail X$, we apply $\phi$ by taking the pull-back and by \cite[Proposition 2.12, 2.13]{Bu} we obtain the identity on the left of the diagarm and then the fact that $\phi(M / X') \cong M$ follows from applying the Five Lemma for exact categories \cite[Corollary 3.2]{Bu} to the diagram 
\[ \begin{tikzcd} X' \arrow[r, tail] \arrow[d, equal] & \phi( M / X') \arrow[d] \arrow[r, two heads] & M / X' \arrow[d, equal] \\ X' \arrow[r, tail] & M \arrow[r, two heads] & M / X'. \end{tikzcd} \] For $N \rightarrowtail X/ X'$, there is a short exact sequence 
\[ \begin{tikzcd} X' \arrow[r, tail] & \phi(N) \arrow[r, two heads] & N. \end{tikzcd} \] Thus, $\phi(N) / X' \cong N$ and we are done. 

Now we show that this is an isomorphism of posets. First we show that if $X' \rightarrowtail M' \rightarrowtail M \rightarrowtail X$ then ${M' / X' \rightarrowtail M / X'}$. This follows from applying \cite[Lemma 3.5]{Bu} to the diagram 
\[\begin{tikzcd} X' \arrow[d, equal] \arrow[r, tail] & M' \arrow[d, tail] \arrow[r, two heads] \arrow[dr, phantom, "\lrcorner", near end] & M' / X' \arrow[d] \\ X' \arrow[r, tail]  & M \arrow[r, two heads] & M / X'.  \end{tikzcd} \]
Finally, we show the converse, that is if $M' / X' \rightarrowtail M / X' \rightarrowtail X/ X'$ then $M' \rightarrowtail M$. From earlier in the proof, there is a commutative diagram \[ \begin{tikzcd}  M' \arrow[d, "\alpha"'] \arrow[r, tail] & M' / X' \arrow[d, tail] \\
  M \arrow[d, tail] \arrow[r, two heads] \arrow[dr, phantom, "\lrcorner", near end] & M / X' \arrow[d, tail] \\  X \arrow[r, two heads] & X / X'. 
\end{tikzcd} \] with the outer rectangle being a pullback. Thus, by the Pullback Lemma and \cite[Proposition 2.15]{Bu}, $\alpha$ is an admissible monic. \end{proof}
\begin{remark}\label{maximal}
By the Fourth $\E$-isomorphism theorem (Proposition \ref{4 iso}), an $\E$-subobject $(Y,f)$ of an object $X$ is $\E$-maximal if and only if for all commutative diagrams \[ \begin{tikzcd} Y \arrow[dr, tail, "f"] \arrow[d, tail, "g"'] & \\  Z \arrow[r, tail, "h"'] & X   \end{tikzcd} \] either $g$ or $h$ is an isomorphism.
\end{remark}
\section{The AI and AIS exact categories}
In abelian categories, the notions of intersection and sum of subobjects are given by pullbacks and pushouts respectively, see \cite[Section 5]{GR} and \cite[Definition 2.6]{pop}. In this paragraph, we investigate whether these concepts carry to exact categories. We recall the definitions of admissible intersection and sum that were first defined by the second author in \cite{HR}, then show that these lead to characterisations of quasi-abelian and abelian categories respectively.

\subsection{Definitions and properties}\label{section:quasi-nice}

The intersection, which exists and is well defined in a pre-abelian exact category, is not necessarily an \emph{admissible} subobject. We recall the definition of exact categories satisfying the admissible intersection property and the admissible sum property from \cite{HR}. Note that, in a previous version of \cite{HR}, the name quasi-n.i.c.e. was used
in the sense that they are {\bf n}ecessarily {\bf i}ntersection {\bf c}losed {\bf e}xact categories, and which we will call {\bf{A.I}} since they admit {\bf{A}}dmissible {\bf{I}}ntersections:

\begin{definition}\cite[Definition 4.3]{HR}\label{quasi-nice}  
An exact category $(\A, 
\E)$ is called an \emph{AI-category} if $\A$ is  pre-abelian additive category  satisfying the following additional axiom:
\begin{itemize}
\item[$({AI})$] 
 The {pullback} $A$ of two admissible monics $j: C \rightarrowtail D$ and $g: B\rightarrowtail D$ exists and yields two admissible monics $i$ and $f$.
\[
\begin{tikzcd}
{A} \arrow[r, tail, "i"] \arrow[d, "f"', tail] & {B} \arrow[d, "g", tail] \\
{C} \arrow[r, tail, "j"']           & {D} \arrow[ul, phantom, "\lrcorner" very near start]          
\end{tikzcd}
\]
\end{itemize}
\end{definition} \label{nice def} 

Let us now introduce a special sub-class of the AI exact categories, that we call {\bf{A.I.S}} exact categories, since they admit {\bf{A}}dmissible {\bf{I}}ntersections and {\bf{S}}ums:
\begin{definition}\cite[Definition 4.5]{HR}\label{AIS} 
An exact category $(\A, \E)$ is called an \emph{AIS-category} if it is an AI-category and moreover it satisfies the following additional axiom:
\begin{itemize}
\item[$({AS})$]The morphism $u$ in the diagram below, given by the universal property of the {pushout} $E$ of $i$ and $f$ {coming from the pullback diagram of the axiom $(AI)$ above}, is an admissible monic.
\[
\begin{tikzcd}
{A} \arrow[r, tail, "i"] \arrow[d, "f"', tail] & {B} \arrow[d, "l", tail]  \arrow[ddr, "g", tail, bend left]& \\
{C} \arrow[r, tail, "k"']   \arrow[drr, tail, "j"', bend right]        & {E} \arrow[ul, phantom, "\ulcorner" near end]  \arrow[dr, tail, "u"] \\ & & D       
\end{tikzcd}
\]
\end{itemize}
\end{definition}

\begin{remark}
One may consider the duals of the above definitions by taking admissible epics instead of monics. Since our focus is on $\E-$subobjects we only study the above and simply remark that the dual definitions lead to the duals of the results of the rest of Section 4, which hold without statement. 
\end{remark}

 Assume now that $(\A,\E)$ is an AIS-category and let us define relative notions of intersection and sum:
\begin{definition}\cite[Definition 4.6]{HR}\label{intersection & sum}
Let $(X_1,i_1)$, $(X_2,i_2)$ be two $\E$-subobjects of an object $X$. We define their \emph{intersection} $X_1{\cap}_X X_2$, to be the pullback
\[ \begin{tikzcd} X_1{\cap}_X X_2  \arrow[dr, phantom, "\lrcorner", very near end]  \arrow[r, "s_1", tail] \arrow[d, "s_2"', tail] & X_1 \arrow[d, "i_1", tail] \\ X_2 \arrow[r, "i_2"', tail] & X. \end{tikzcd} \] 
We then define their \emph{sum}, $X_1{+}_{X}X_2$, to be the pushout  
\[ \begin{tikzcd} X_1{\cap}_X X_2    \arrow[r, "s_1", tail] \arrow[d, "s_2"', tail] & X_1 \arrow[d, "j_1", tail] \\ X_2 \arrow[r, "j_2"', tail] & \arrow[ul, phantom, "\ulcorner", near end]X_1{+}_{X}X_2. \end{tikzcd}  \]
\end{definition}

\begin{remark} \label{sum intersection as coker ker} Equivalently, for two $\E$-subobjects $(X_1,i_1)$, $(X_2,i_2)$  of an object $X$ we have 
\[
    X_1{\cap}_X X_2  = \Ker \left( \begin{tikzcd}[sep= large]
X_1 \oplus X_2 \arrow[r, "{[ i_1 - i_2] }"] & X \end{tikzcd} \right) \]
and 
\[
    X_1{+}_{X}X_2  = \Coker \left( \begin{tikzcd}[sep=large]
X_1 \cap_X X_2 \arrow[r, "{\left[ s_1  -s_2  \right]^t}"] & X_1 \oplus X_2
\end{tikzcd} \right). \]
\end{remark}

Let us note that these definitions generalises the abelian versions as shown in \cite{HR}.
\begin{remark}\cite[4.8, 4.12, 4.13]{HR}\label{HR}   Let  $(X_1,i_1)$, $(X_2,i_2)$ and $(Y,j)$ be $\E$-subobjects of an object $X$. Then
\begin{enumerate}
    \item[a)]$X_1 \cap_X X_1= X_1 = X_1 +_X X_1$.
    \item[b)]If $X_1{+}_{X}X_2=0_{\A}$ then $X_1=X_2=0_{\A}$.
     \item[c)] If $(\A, \E)$ is an AI-category and there exists an admissible monic \[ i : X_1\rightarrowtail X_2\] then there exists an admissible monic \[X_1{\cap} Y\rightarrowtail X_2{\cap} Y.\]
     \item[d)] If $(\A, \E)$ is an AIS-category and there exists an admissible monic \[ i : X_1\rightarrowtail X_2\] then there exists an admissible monic \[X_1+ Y\rightarrowtail X_2+ Y.\]
\end{enumerate}
\end{remark}

\begin{lemma} \label{zero intersection}
Let $(\A, \E)$ be an exact category and let $f:X \rightarrowtail Z$ and $g: Y \rightarrowtail Z$ be admissible monics. Suppose that $X \cap_Z Y$ exists and is the zero object, then $X+_Z Y \cong X \oplus Y$.
\end{lemma}
\begin{proof}
By assumption, there is a pullback diagram in $\A$:
\[\begin{tikzcd}
 0 \arrow[r, tail] \arrow[d, tail] & X \arrow[d, "f", tail] \\ Y \arrow[r, tail, "g"'] & Z. \arrow[ul, phantom, "\lrcorner", near start] 
\end{tikzcd}\] By direct computation we have that 
\[\begin{tikzcd}
 0 \arrow[r, tail] \arrow[d, tail] & X \arrow[d, "s", tail] \\ Y \arrow[r, tail, "t"'] & X \oplus Y \arrow[ul, phantom, "\ulcorner", near end] 
\end{tikzcd}\] is a pushout diagram for any {pair of morphisms  $s$ and $t$ satisfying the universal property of the coproduct}. Thus, by definition, $X+_Z Y \cong X \oplus Y$. 
\end{proof}

\subsection{AI-categories and quasi-abelian categories}
It is not difficult to see that the split exact structure $\E_{min}$ does \emph{not} satisfy axiom (AI) unless every sequence splits in $\A$. Compare also 
\cite[Remark 2.4 and 5.3]{Kelly} which helps to show that the category of abelian groups equipped with $\E_{min}$ does \emph{not} satisfy axiom (AI).
In fact, an exact structure needs to contain {\em all} short exact sequences in order to satisfy the (AI) axiom:

\begin{proposition}\label{all} Let $(\A ,\E)$ be an exact category. If $(A, \E)$ is  an AI-category, then $\E = \E_{all}$.
\end{proposition}

\begin{proof}
Let us suppose that the exact structure $\E$ is strictly included in $\E_{all}$, thus there exists a short exact sequence \[ S : \qquad \begin{tikzcd} 0 \arrow[r] & L \arrow[r, "f"] & M \arrow[r, "g"] & N \arrow[r] & 0 \end{tikzcd} \] 
such that $S \notin \E$.
\bigskip

Consider the two sections
\[
\bsm 1 \\ g \esm: M\rightarrow M\oplus N\]
and
\[
\bsm 1 \\ 0 \esm: M\rightarrow M\oplus N.\]
\\
It is easy to verify that the pull-back of these two morphisms is:
\[
\begin{tikzcd}
  L \arrow[r, tail, "f"] \arrow[d, tail,"f"'] & M \arrow[d, "{\bsm 1 \\ g \esm}", tail] \\ M \arrow[r, tail, "{\bsm 1 \\ 0 \esm}"'] & M \oplus N. \arrow[ul, phantom, "\lrcorner", near start] 
\end{tikzcd}
\]

Since $f$ is not admissible in $\E$, the (AI) axiom is not satisfied and 
$(\A, \E)$ is therefore not an AI-category.

\end{proof}
\begin{remark}\label{unique}
The previous proposition shows that an exact structure satisfying the (AI) axiom is unique, when it exists on an additive category.
\end{remark}

Before showing that the AI categories form a sub-class of quasi-abelian additive categories, let us recall the following:

\begin{lemma}\cite[1.1.7]{Sch}\cite[4.4]{Bu}\label{qa 1} In any quasi-abelian category, the class of all short exact sequences defines an exact structure $\E_{all}$ and this is the maximal one $\E_{max}=\E_{all}$. In particular this is the case for abelian categories (see also \cite{Ru01}).
\end{lemma}

\begin{lemma}\label{qa 2}
Every additive category $\A$ admitting $\E_{all}$ as an exact structure is quasi-abelian.
\end{lemma}
\begin{proof}
By the axioms (A2) and (A2)$^{op}$ of an exact structure, every short exact sequence is stable. So it follows from Definition \ref{quasi abelian} that
$\A$ is quasi-abelian.
\end{proof}

\begin{theorem}\label{quasi-nice is qa}
Every AI-category $\A$ is quasi-abelian.
\end{theorem}
\begin{proof}
By Proposition \ref{all}, every exact structure satisfying the (AI) axiom is $\E= \E_{all}$ and then, by Lemma \ref{qa 2} $\A$, is quasi-abelian.
\end{proof}

\begin{example}\label{qa not nice}
We provide an example which demonstrates that not every quasi-abelian additive category with its maximal exact structure admits admissible sums: 
Consider the quiver
\[ Q : \qquad 1 \longrightarrow 2 \longrightarrow 3 
\]
The Auslander-Reiten quiver of $\rep Q$ is as follows:

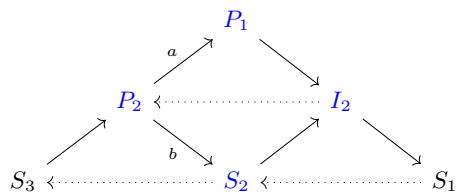
\[\begin{tikzcd}
               &                                                    & {\color{blue}P_1} \arrow[rd]                    &                                                 &                        \\
               & {\color{blue}P_2} \arrow[ru, "a"] \arrow[rd, "b"'] &                                                 & {\color{blue}I_2} \arrow[rd] \arrow[ll, dotted] &                        \\
S_3 \arrow[ru] &                                                    & {\color{blue}S_2} \arrow[ru] \arrow[ll, dotted] &                                                 & S_1 \arrow[ll, dotted]  
\end{tikzcd}\]
Let {\color{blue}$\A$} be the full additive subcategory generated by the indecomposables {\color{blue}$P_2,P_1,S_2,I_2.$} Then $\A$ is an intersection $\cF \cap \T'$ where { $\cF$ is the torsion free class  of the hereditary torsion pair $(\T,\cF)= (\add(S_1), \add(S_3 \oplus P_2 \oplus P_1 \oplus S_2 \oplus I_2))$ and $\T'$ is the torsion class  of the cohereditary torsion pair $(\T',\cF')= (\add(P_2 \oplus P_1 \oplus S_2 \oplus I_2 \oplus S_1), \add(S_3))$ of $\rep Q$}. Following \cite[Theorem 3.2]{T19} and  \cite[Theorem 2]{Ru01}, we conclude that $\A$ is an integral quasi-abelian category with exact structure $\E_{all}$ generated by the Auslander-Reiten sequence
\[ \begin{tikzcd} 0 \arrow[r] & P_2 \arrow[r] & P_1 \oplus S_2 \arrow[r] & I_2 \arrow[r] & 0 \end{tikzcd}
\]

We verify that the axiom (AS) fails; to that end, we consider the following admissible monics in $\A$:
\[
\begin{tikzcd}
 & P_2 \arrow[d, tail, "{\bsm a \\ b\esm}"] \\ P_1 \arrow[r, tail, "{\bsm 1 \\ 0 \esm}"'] & P_1 \oplus S_2
\end{tikzcd}
\]
The pullback along these monics in the abelian category $\rep Q$
is given by the object $S_3$, but this is not available in $\A.$ 
Being quasi-abelian and so pre-abelian, $\A$ admits a pullback which is a subobject of the abelian pullback, thus the zero object. Hence, we have in $\A$ that 
the intersection along the given monics is 
$P_1 \cap P_2 = 0$, and therefore, by Lemma \ref{zero intersection}, $P_1 + P_2 = P_1 \oplus P_2$. However, the direct sum $P_1 \oplus P_2$ is not an admissible subobject of $P_1 \oplus S_2$, thus the axiom (AS) fails.
\end{example}
\medskip

The next results shows that having admissible intersections is not enough to be abelian. The computation of the previous example suggests to look for a quasi-abelian subcategory of an abelian category where the pullback diagrams in axiom (AI) coincide with the abelian ones (this is not in general true for all kernels), thus one gets an  AI-subcategory.
Typical examples of such quasi-abelian but non-abelian categories arise in functional analysis (see \cite{HSW} for more examples):

\begin{definition}
We denote by $\Ban$ the category of Banach spaces (over the field of real numbers). The objects of $\Ban$ are the complete normed $\bR-$vector spaces, and morphisms are continuous linear maps.
\end{definition}

The kernel of a morphism $f:X \to Y$ in $\Ban$ is the linear kernel
$ f^{-1}(0) \rightarrowtail X $, however the cokernel \[Y \twoheadrightarrow Y/ \overline{f(X)}\] in $\Ban$ is in general different from the linear cokernel $Y \to Y/ f(X)$. 
Thus $f:X \to Y$ is an admissible monic in $\Ban$ precisely when $f$ is a monomorphism such that $f(X)$ is closed in $Y$.
The Open Mapping Theorem for Banach spaces guarantees that an admissible monic $f:X \to Y$ is an isomorphism onto $f(X)$.
In fact the class $\E=\E_{all}$ of all kernel-cokernel pairs coincides with the class of short exact sequences of bounded linear maps, see \cite[IV.2]{Bu11}.
It is well-known that the category $\Ban$ is quasi-abelian with the maximal exact structure $\E_{all}$, but it is not abelian. 
We verify here the admissible intersection property and we reprove, using Theorem \ref{quasi-nice is qa} that $\Ban$ is quasi-abelian:

\begin{theorem}\label{nice not abelian}
The category $\Ban$ of Banach spaces,  equipped  with the maximal exact structure $\E=\E_{all}$, is an AI-category.
\end{theorem}
\begin{proof}
Consider two $\E$-subobjects $(X_0,f_0)$, $(X_1,f_1)$ of an object $X$ in $\Ban$. Since the admissible monics $f_i$ are isomorphisms onto their range $f_i(X_i)$, we can identify $X_0$ and $X_1$ with closed subspaces of $X$. The intersection of closed subspaces is closed, therefore we have the following diagram of closed embeddings (which are admissible monics):
\[
 \begin{tikzcd}
{X_0{\cap} X_1} \arrow[r, tail, "i_1"] \arrow[d, "i_0"', tail] & {X_1} \arrow[d, "f_1", tail]   \\
{X_0} \arrow[r, tail, "f_0"']        & {X}    
\end{tikzcd}
\]
From \cite{HR}, we know that the object $X_0 \cap X_1$  satisfies the pullback property from axiom (AI) in $\mod \bR.$
Since the pullback can be written as kernel (Remark \ref{sum intersection as coker ker}) and kernels in $\Ban$ are the kernels in mod $\bR,$ we conclude that the (AI)-axiom is satisfied: The pullback along admissible monics exists, and yields admissible monics. 
\end{proof}

\begin{remark}
While $\Ban$ satisfies the admissible intersection property, it does not satisfy the admissible sum property and so it is not abelian by Theorem \ref{abelian 1}.
Consider for a moment the second axiom (AS) in the setting studied in the proof of the previous theorem: It is shown in \cite[Chapter 3.1]{BS} that {\em both,} the intersection $X_0 \cap X_1$ and the sum $X_0 + X_1$ (as subvector spaces of $X$) admit norms turning them into Banach spaces, satisfying that
\[  X_0 \cap X_1 \hookrightarrow X_i \hookrightarrow X_0 + X_1 
\]
are continuous embeddings for $i=0,1.$ In fact, the whole interval between $X_0 \cap X_1$ and $ X_0 + X_1$ is studied in \cite{Ma}, as {\em interpolations} between intersection and sum. 
We summarize the situation in the following diagram:
\[
\begin{tikzcd}
{X_0{\cap} X_1} \arrow[r, tail, "i_1"] \arrow[d, "i_0"', tail] & {X_1} \arrow[d, "j_1", tail]  \arrow[ddr, "f_1", tail, bend left]& \\
{X_0} \arrow[r, tail, "j_0"']   \arrow[drr, tail, "f_0"', bend right]        & {X_0{+} X_1}  \arrow[dr, "r"] \\ & & X       
\end{tikzcd}
\] 
The sum $X_0+X_1$  is the pushout in $\mod \bR$, hence satisfies the pushout property in $\Ban$ since the kernel-cokernel pairs of bounded maps in mod $\bR$ are also exact in $\Ban$. However,  the inclusion map $r : X_0+X_1 \to X$ (which is bounded, thus continuous) is not an admissible monic in general: The norm on $X_0+X_1$ is given in \cite[Chapter 3.1]{BS} by 
\[ \| x \|_{X_0+X_1} = \inf \{ \| x_0 \|_{X_0} + \| x_1 \|_{X_1} \; | \; x_0+x_1 = x \}, \]
and with respect to this norm, the subspace  $X_0+X_1$ in $X$ is not necessarily closed.
In fact, since we show, in Proposition \ref{nice is abelian}, that any AIS-category is abelian, and $\Ban$ is not abelian, we conclude that {\em it is impossible to define a norm on the subspace $X_0+X_1$ of $X$ which turns $X_0+X_1$ into a complete closed subspace of $X$.}
Rephrased in terms of the poset $\P^\E_X$ of closed subspaces of a Banach space $X$, the observation above shows that the sum $X_0+X_1$ is  in general not an element of $\P^\E_X$. However, the intersection $X_0 \cap X_1$ is defined in $\P^\E_X$, turning the poset $\P^\E_X$ into a meet semi-lattice. It is well-known that a meet semi-lattice is a lattice when it is {\em complete}, i.e. closed under arbitrary intersections, and admits a unique maximal element (which is $X$ here). However, this is also not true for the category $\Ban$ equipped with $\E=\E_{all}$ since an infinite intersection of closed subspaces need not be closed.
\end{remark}

\bigskip 

We proved in Theorem \ref{quasi-nice is qa} that admitting admissible intersections is enough to be quasi-abelian, but it has been proved recently in \cite[Theorem 6.1]{HSW}\label{HSW 6.1} that the converse also holds, and hence together a new characterisation of quasi-abelian categories is established. 
For the convenience of the reader, and with the kind permission of the authors of \cite{HSW}, we also include their part of the proof below:

\begin{theorem}{\bf{(Br\"ustle, Hassoun, Shah, Tattar, Wegner)}}
A category $(\mathcal{A}, \E_{max})$ is quasi-abelian if and only if it is an AI-category.
\end{theorem} 
\begin{proof}

$(\Longleftarrow)$\; By Theorem \ref{quasi-nice is qa} every AI-category is quasi-abelian. 

$(\Longrightarrow)$\; Let $\A$ be a quasi-abelian category. 
Endowing it with the class $\E$ of all kernel-cokernel pairs in $\A$ yields an exact category $(\A,\E)$ as $\A$ is quasi-abelian; see \cite[Rmk.\ 1.1.11]{Sch}. The class of admissible monomorphisms in $(\A,\E)$ is thus precisely the class of kernels in $\A$. Let $c\colon B \rightarrowtail D$ and $d\colon C \rightarrowtail D$ be arbitrary admissible monomorphisms in $(\A,\E)$, i.e.\ $c,d$ are kernels. Then in the pullback diagram 
\[
\begin{tikzcd}
A\arrow{r}{a}\arrow{d}[swap]{b}\arrow[dr, "\ulcorner" near start, phantom] & B \arrow[tail]{d}{c}&\\
C \arrow[tail]{r}[swap]{d} & D 
\end{tikzcd}
\]
the morphisms $a$ and $b$ are also kernels in $\A$ by the dual of Kelly \cite[Prop.\ 5.2]{Kelly}. That is, $a,b$ are admissible monomorphisms, and we see that $(\A, \E)$ has admissible intersections.

\end{proof}

\subsection{AIS-categories and  abelian categories}
In this subsection we prove that the categories satisfying both the (AI) and the (AS) axioms are exactly the abelian categories.

\begin{proposition}\label{P is lattice}
 Let $X$ be an object in an  AIS-category $(\A, \E)$. Then the poset $\mathcal{P}^\E_X$ forms a lattice under the operations
\[(\mathcal{P}^\E_X,\leq, {\cap}_X,  {+}_X)\] where $\cap_X$ and ${+}_X$ are the intersection and sum operations defined in Definition \ref{intersection & sum}.
\end{proposition}
\begin{proof}
We need to verify the axioms of Definition \ref{lattice}.
 The first and second axioms follow directly from Remark \ref{sum intersection as coker ker}.  For the third axiom, we have to show  
\[ Y+(Y \cap Z) = Y = Y \cap (Y+Z)
\]
for $\E-$subobjects $Y,Z$ of $X$. We give the proof of the first equality here, the second one being similar:
By axiom (AI), we know that there is an admissible monic 
$Y \cap Z \rightarrowtail Y.$ Remark \ref{HR} applied to this inclusion and the object $Y$ yields
\[ (Y \cap Z)+Y \rightarrowtail Y+Y.
\]
Since $Y+Y=Y$ by Remark \ref{HR}, we have an admissible monic 
$Y+(Y \cap Z) \rightarrowtail Y$. But by the axiom (AS), there is also an admissible monic from $Y$ into the sum of $Y$ with $Y \cap Z$, therefore we have 
\[ Y \rightarrowtail (Y \cap Z)+Y \rightarrowtail Y.
\]
This shows that the monics are isomorphisms, hence equalities  in the poset $\P^\E_X.$
\end{proof}

\begin{lemma}\cite[Corollary 4.11]{HR}\label{abelian is nice}
Let $\A$ be an abelian category. Then $(\A,\E_{all})$ is an AIS-category.
\end{lemma}

\begin{lemma} \label{mono = kernel}
Let $\A$ be a quasi-abelian category and $\E= \E_{all}$. Suppose that every monomorphism in $\A$ is a kernel, then $\A$ is abelian. Dually, if every epimorphism is a cokernel, then $\A$ is abelian.
\end{lemma}
\begin{proof}
Let $f: X \to Y$ be an arbitrary morphism in $\A$ we will show that $f$ is admissible, whence it follows that $\A$ is abelian by {Lemma \ref{F66}}. Recall that there is a commutative diagram in $\A$
\[ \begin{tikzcd}
 \Ker f \arrow[d, tail] & \Coker f \\ X \arrow[r, "f"] \arrow[d, "c", two heads] & Y \arrow[u, two heads] \\ \Coim f \arrow[r, "\bar{f}"] & \Im f \arrow[u, "i", tail] 
\end{tikzcd}\] where $\bar{f}$ is both monic and epic (see \cite[Section 1]{Ru01} for details) and the columns are $\E$-sequences since $\A$ is quasi-abelian and $\E$ = $\E_{all}$. By assumption, the composition $i\bar{f}$ is a kernel and therefore an admissible monic since $\A$ is quasi-abelian. Thus the decomposition  {$f=(i\bar{f})c$} shows that $f$ is admissible. The proof of the dual statement is similar. 
\end{proof}

\begin{proposition}\label{nice is abelian}
Let $(\A, \E)$ be an AIS-category. Then $\A$ is abelian and $\E= \E_{all}$.
\end{proposition}
\begin{proof}
By Theorem \ref{quasi-nice is qa}, $\A$ is quasi-abelian and $\E= \E_{all}$. Thus, by Lemma \ref{mono = kernel}, it is enough to show that every monomorphism $f:X \to Y$ in $\A$ is a kernel. To this end, consider the $\E$-subobjects given by two sections
\begin{align*}
    \bsm 1 \\ f \esm: &X \rightarrowtail X \oplus Y \\ \bsm 1 \\ 0 \esm: & X \rightarrowtail X \oplus Y.
\end{align*}
By computation their intersection is the zero-object
\[ \begin{tikzcd}
  0 \arrow[r, tail] \arrow[d, tail] & X \arrow[d, "{\bsm 1 \\ f \esm}", tail] \\ X \arrow[r, tail, "{\bsm 1 \\ 0 \esm}"'] & X \oplus Y. \arrow[ul, phantom, "\lrcorner", near start] 
\end{tikzcd}\] Thus, by Lemma \ref{zero intersection}, their sum is given by the direct sum $X \oplus X$
\[\begin{tikzcd}
 0 \arrow[r, tail] \arrow[d, tail] & X \arrow[d, "{\bsm 1 \\ 1 \esm}", tail] \arrow[ddr, tail, bend left, "{\bsm 1 \\ f \esm}"] & \\ X \arrow[r, tail, "{\bsm 1 \\ 0 \esm}"'] \arrow[drr, tail, bend right, "{\bsm 1 \\ 0 \esm}"'] & X \oplus X  \arrow[ul, phantom, "\ulcorner", near end] \arrow[dr, "u", tail] & \\ & & X \oplus Y 
\end{tikzcd}\] where $u = \bsm 1 & 0 \\ 0 & f \esm$ is an admissible monic since $(\A, \E)$ is (AIS). Now, by \cite[Corollary 2.18]{Bu}, $f$ is an admissible monic and we are done.  
\end{proof}
\begin{theorem}\label{abelian 2}
An exact category $(\A, \E)$ is an AIS-category if and only if $\A$ is abelian and $\E=\E_{all}$.
\end{theorem}
\begin{proof}
By using Lemma \ref{abelian is nice} and Proposition \ref{nice is abelian}.
\end{proof}
Now that we know that the AIS-categories are preciesly the abelian ones with their maximal exact structure, let us recall the second isomorphism theorem for an abelian category with its maximal exact structure, but using the language of exact categories where the intersection and sum defined in Definition \ref{intersection & sum} by {pullbacks and pushouts} are always admissible. This result will be useful later to prove that the abelian categories are diamond exact categories in the sense of Definition \ref{Diamond} (we refer the reader to \cite[Lemma 5.2]{HR} for the proof):

\begin{proposition} ({\bf{The second  $\E-$isomorphism theorem}})\label{parallelo} 
Assume that $(\A,\E)$ is an AIS-exact category. Let $X$ and 
$Y$ be  $\E-$subobjects of an object $Z$ in $\A$. 
Then there is the following  short exact sequence:
\[X{\cap}_{Z}Y\rightarrowtail Y\twoheadrightarrow (X{+}_Z Y)/X.\]
In other words, there is an isomorphism ({parallelogramm} identity):
\[ Y / (X{\cap}_{Z}Y) \cong (X{+}_Z Y)/X.\]
\end{proposition}


\section{The diamond exact categories}
In this section we address the drawbacks of the intersection and sum in the previous section by introducing a general notion of intersection and sum that applies to exact categories. We then use this to introduce a class of exact categories - the diamond exact categories - and show that these satisfy the $\E-$Jordan-H\"older property as in Definition \ref{composition series}. 

\subsection{Jordan-H\"older property}
\begin{definition}\label{composition series}
Let $(\A, \E)$ be an exact category. A finite $\mathcal{E}-$composition series for an object $X$ of $\mathcal{A}$ is a sequence 
\begin{equation}\label{chain1}
    \begin{tikzcd}[sep=large] 0= X_{0} \arrow[r, tail, "i_0"] &  X_1 \arrow[r, tail, "i_1"] & \dots \arrow[r, tail, "i_{n-2}"] &  X_{n-1} \arrow[r, tail, "i_{n-1}"] & X_n = X  \end{tikzcd}
\end{equation} 
where all $i_l$ are \emph{proper admissible monics} with $\E-$simple cokernel. 
We say an exact category $(\A,\E)$ has the {\em ($\E-$)Jordan-H\"older property} or is a {\em Jordan-H\"older exact category}
 if any two finite $\mathcal{E}-$composition series for an object $X$ of $\mathcal{A}$
\[ 
 \begin{tikzcd}[sep=large] 0= X_{0} \arrow[r, tail, "i_0"] &  X_1 \arrow[r, tail, "i_1"] & \dots \arrow[r, tail, "i_{m-2}"] &  X_{m-1} \arrow[r, tail, "i_{m-1}"] & X_m = X  \end{tikzcd}\] 
and
\[
 \begin{tikzcd}[sep=large] 0= X'_{0} \arrow[r, tail, "i'_0"] &  X'_1 \arrow[r, tail, "i'_1"] & \dots \arrow[r, tail, "i'_{n-2}"] &  X'_{n-1} \arrow[r, tail, "i'_{n-1}"] & X'_n = X  \end{tikzcd}\]
are equivalent, that is, they have the same length and the same composition factors, up to permutation and isomorphism. 
\end{definition}

\begin{remark} \label{niceJH} As shown in \cite{HR}, one can use the same steps as in \cite{Bau} and the $\E-$Schur lemma to prove that every AIS-category  $(\mathcal{A}, \mathcal{E})$  is a Jordan-H\"older exact category. 
\end{remark}

\subsection{General intersection and sum}

For an AIS-category $(\A,\E)$, or equivalently, for an abelian category $\A$ with maximal exact structure $\E_{all}$, the intersection of two subobjects of $X$ is defined as the {pullback} of their monomorphisms in $X$ and their sum is defined as the {pushout of this pullback}, which is also admissible.
In terms of the poset $\P^\E_X$ of $\E-$subobjects of $X$,  this means that $\P^\E_X$ forms a lattice as shown in Proposition \ref{P is lattice}.
However, in general the poset $\P^\E_X$ is not a lattice, even when the $\E-$Jordan-H\"older property holds for the exact category $(\A,\E)$, as the following simple examples demonstrate.

\begin{example}\label{even dimensional spaces}
Let $\A$ be the category of all even dimensional $k$-vector spaces endowed with the split exact structure $\E=\E_{min}$. Then the $\E-$simple objects are precisely the two-dimensional vector spaces, and the Jordan-H\"older property is clearly satisfied. 
Consider the object $X=k^6$ with basis $\{  v_1,v_2,v_3,v_4,v_5,v_6 \}$ and the two elements of $\P^\E_X$ given by \[
V_1 = < v_1,v_2,v_3,v_4> \quad \mbox{and} \quad V_2 = <v_2,v_3,v_4,v_5>. 
\]
The intersection $V_1 \cap V_2\;$ in $\mod k$ is $V_3 =  <v_2,v_3,v_4>$. But since $V_3$ is not in $\A$, every two-dimensional subspace $U$ of $V_3$ is a maximal lower bound for both $V_1$  and $V_2$, when we view $(U, f)$  as an element in $\mathcal{P}^\E_X$ with its inclusion map $f$.
Therefore $\P^\E_X$ is not a lattice, and the intersection of $V_1$  and $V_2$ is not unique in $(\A,\E)$, in fact it is an infinite set formed by all embeddings $(U,f)$ of maximal proper subspaces $U$ of $V_3$. 
\end{example}

\begin{example}\label{example A2 again}
A similar phenomenon can be observed studying the additive category $\A=\rep A_2$ of representations of the quiver of type $A_2$, endowed with the minimal exact structure $\E=\E_{min}$.
We denote the simple representations by $S_1$ and $S_2$, and the indecomposable projective-injective representation by $P_1$. Then  there is a non-split indecomposable short exact sequence in $\A$
\[
\begin{tikzcd}
 0 \arrow[r] & S_2 \arrow[r, "f"] & P_1 \arrow[r, "g"] & S_1 \arrow[r] &0 
\end{tikzcd}\]
 which is not admissible in $\E_{min}$.
Therefore
$(\A , \E_{min})$ is not an AI-category   by Proposition \ref{all}.
Choosing the object $X=S_2 \oplus P_1 \oplus S_1,$ we observe that
there are many maximal $\E-$subobjects of $X$ with quotient $S_1$ given by $(S_2 \oplus P_1, \alpha_\lambda) $ with $\lambda \in k$, where
\[ \bsm
1 & 0 \\
0 & 1 \\
0 & \lambda g
\esm = \alpha_\lambda : S_2 \oplus P_1 \rightarrow  X = S_2 \oplus P_1 \oplus S_1.\]
Each of these admit many maximal $\E-$subobjects with quotient $P_1$ given by $(S_2 \oplus P_1, \beta_\mu )$ with $\mu \in k$ where
\[ \bsm
1  \\
\mu f
\esm = \beta_\mu : S_2  \rightarrow  S_2 \oplus P_1.\]
\end{example}
\bigskip

The preceding examples motivate the following definition, where we allow the (generalised) intersection and sum to be a set of objects:
\begin{definition}\label{general intersection and sum} {
Let $(A_i, f_i)$, $i \in I$, be a collection of $\E-$subobjects of $X$ indexed by a set $I$. We denote the set of all their common admissible subobjects with respect to $X$ as 
\[ \Sub_X (\{ (A_i, f_i) \mid i \in I\} ) := \{ (Y,h) \in P^\E_X\; \mid \;  (Y, h) \in \P^\E_{A_i}; \forall i \in I \}  \]
 and define the $\E-$relative intersection of the $(A_i, f_i)$ in  $\P^\E_X$ as 
\[\inter_X(\{ (A_i, f_i) \mid i \in I\} ) := \Max(\Sub_X(\{ (A_i, f_i) \mid i \in I\}),
\]
 the set of maximal elements in $\Sub_X(\{ (A_i, f_i) \mid i \in I\})$  (where we define the generalised intersection over the empty set to be $\{0\}$). Dually, we denote the set of all common superobjects of $A_i, i\in I$
 \[ \Sup_X(\{ (A_i, f_i) \mid i \in I\}) := \{ \;
(Y,h) \in \P^\E_X \; \mid \;  (A_i, f_i) \in \P^\E_Y, \; \forall i \in I
\} 
\]
and define the $\E-$relative sum of the $A_i, i\in I$ in $\P^\E_X$ as 
\[\Sum_X(\{ (A_i, f_i) \mid i \in I\}) := \Min(\Sup_X(\{ (A_i, f_i) \mid i \in I\})),
\]
the set of minimal elements in $\Sup_X(\{ (A_i, f_i) \mid i \in I\})$. }
\end{definition}
\begin{example}
In the setup of Example \ref{even dimensional spaces}, the objects $V_1$ and $V_2$ have as $\E-$relative intersection  in $\P^\E_X$ the Grassmannian
$\inter_X(V_1,V_2)=Gr(2,3)$ of  all maximal proper subspaces of $V_3$.
The set $\Sum_X(V_1,V_2)$ however consists only of the element $X$ itself.
In Example \ref{example A2 again}, any two of the objects $(S_2 \oplus P_1, \alpha_\lambda) $ have an infinite intersection containing all elements $(S_2, \beta_\mu) $ of $\P^\E_X$, and conversely, any two of the $(S_2, \beta_\mu) $ have an infinite sum containing all the  objects $(S_2 \oplus P_1, \alpha_\lambda) $.
\end{example}

\subsection{The diamond categories are Jordan-H\"older exact categories}

In this section we prove the $\E-$Jordan-H\"older property in a more general context than abelian categories, namely for exact categories that we call \emph{the diamond exact categories}:
\begin{definition}({\bf{Diamond Axiom}})\label{Diamond}
Let $(A,f)$ and $(B,g)$ be two distinct maximal $\E-$sub-objects in $\P_X$, that is, their cokernels $X/A$ and $X/B$ are $\E-$simple. 
We say that $(A,f)$ and $(B,g)$ satisfy the {\em diamond axiom} if for every $Y \in \inter_X(A,B)$ we have that $A/Y$ and $B/Y$ are both $\E-$simple and {the elements of the sets
 $\{X/A, A/Y \}, \{ X/B, B/Y \}$ are equal up to permuation and isomorphism.}
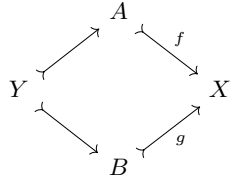
\[ \begin{tikzcd} & A \arrow[dr, tail, "f"] & \\  Y \arrow[dr, tail ] \arrow[ur, tail] &  & X \\ & B \arrow[ur, tail, "g"'] &   \end{tikzcd}
\]
A \emph{diamond} exact category $(\A,\E)$ is an exact category that satisfies the diamond axiom for each pair of maximal subobjects $A$ and $B$ {of a fixed object $X$}.
\end{definition}

\begin{remark}\label{ab is diamond}
When $\A$ is an abelian category, then for each object $X$  we have that $\inter_X(A,B)$ and $\Sum_X(A,B)$ are given by the unique objects $A \cap_X B$ and $A +_X B$, respectively. Lemma \ref{parallelo} then ensures that the diamond axiom is always satisfied. We conclude that every abelian category is a diamond exact category.
\end{remark}
\begin{remark}
Note that Lemma \ref{parallelo} guarantees that we always have crosswise isomorphisms
\[  X/A \cong B/Y \quad \mbox{and} \quad X/B \cong A/Y\]
in the abelian case. However, Example \ref{example A2 again} shows that one can have the lengthwise isomorphisms \[  X/A \cong A/Y \quad \mbox{and} \quad X/B \cong B/Y\]
when the poset $\P^\E_X$ is not a lattice.
\end{remark}

\begin{lemma}\label{D2}
Assume that an object $X$ in a diamond exact category $\A$ has a composition series of length $n$
\[
\begin{tikzcd}0=B_0 \arrow[r, tail] & B_1 \arrow[r, tail] & \dots \arrow[r, tail]            & B_n=X. \end{tikzcd}\] 
If $(C,f)$ is a maximal element in $\P_X$ , then there exists a composition series of $X$ through $C$ of length $n$ :
\[\begin{tikzcd}0=C_0 \arrow[r, tail] & C_1 \arrow[r, tail] & \dots \arrow[r, tail]            & C_{n-2} \arrow[r, tail] & C \arrow[r, "f"] & X. \end{tikzcd}
\]
\end{lemma}
\begin{proof} By induction on $n.$
For $n=1$, this is obvious because $C=0$.
Assume now $n \ge 2$. If $B_{n-1} = C$ as elements in $\P^\E_X$, we can use the given composition series of $X$. Otherwise, consider an element $Y \in \inter_X(B_{n-1},C)$:
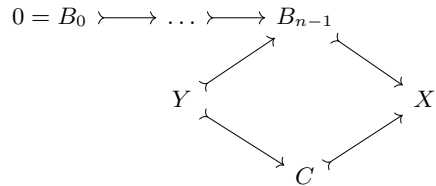
\[
\begin{tikzcd}
0=B_0 \arrow[r, tail] & \dots \arrow[r, tail]               & B_{n-1} \arrow[rd, tail] &   \\
                      & Y \arrow[ru, tail] \arrow[rd, tail] &                          & X \\
                      &                                     & C \arrow[ru, tail]       &  
\end{tikzcd}
\]
By the diamond axiom,  both quotients $B_{n-1}/Y$ and $C/Y$ are $\E-$simple since $B_{n-1}$ and $C$
 are maximal elements in $\P^\E_X$.
Thus we have a composition series of length $n-1$
\[\begin{tikzcd}0=B_0 \arrow[r, tail] & B_1 \arrow[r, tail] & \dots \arrow[r, tail]            & B_{n-1}=X' \end{tikzcd}
\]
and $Y$ is maximal in $\P^\E_{X'}.$
Induction hypothesis implies that  there exists a composition series of $X'$ through $Y$ of length $n-1$. 
Replacing $Y \rightarrowtail B_{n-1}$ in this series by $Y \rightarrowtail C \rightarrowtail X$ yields a composition series of $X$ through $C$ of length $n$:
\[\begin{tikzcd}0=Y_0 \arrow[r, tail] & \dots \arrow[r, tail] & Y_{n-3} \arrow[r, tail] & Y \arrow[r, tail]            &  C \arrow[r, "f"] & X. \end{tikzcd}
\] \end{proof}

\begin{theorem}\label{JH}
Every diamond exact category is a Jordan-H\"older exact category. 
\end{theorem}
\begin{proof}
Following the strategy of the proof in \cite[Chapter 4.5]{Pa},
assume we are given two composition series
\[\begin{tikzcd}0=B_0 \arrow[r, tail] & B_1 \arrow[r, tail] & \dots \arrow[r, tail]            & B_n=X \end{tikzcd}
\]
and 
\[ \begin{tikzcd}0=C_0 \arrow[r, tail] & C_1 \arrow[r, tail] & \dots \arrow[r, tail]            & C_m=X. \end{tikzcd}
\]
We proceed by induction on $n.$
For $n=1$, the object $X$ is $\E-$simple and the statement clearly holds.
Assume now $n \ge 2$.
For any object $Y \in \inter_X(B_{n-1},C_{m-1})$ we obtain the following diagram:
\[
\begin{tikzcd}
0=B_0 \arrow[r, tail] & \dots \arrow[r, tail]               & B_{n-1} \arrow[rd, tail] &   \\
                      & Y \arrow[ru, tail] \arrow[rd, tail] &                          & X \\
0=C_0 \arrow[r, tail] & \dots \arrow[r, tail]               & C_{m-1} \arrow[ru, tail] &  
\end{tikzcd}
\]
The diamond axiom applied to the maximal $\E-$subobjects $B_{n-1},C_{m-1}$ of $X$ yields that $Y$ is maximal in both $B_{n-1}$ and $C_{m-1}$.
Lemma \ref{D2} applied to the maximal element $Y$ of $B_{n-1}$
yields a composition series
\[\begin{tikzcd}
 0=Y_0 \arrow[r, tail] & \dots \arrow[r, tail] & Y_{n-3} \arrow[r, tail]& Y \arrow[r, tail] & B_{n-1}
\end{tikzcd}
\] of length $n-1$.
 Moreover, Lemma \ref{D2} applied to the maximal element $Y$ of $C_{m-1}$
yields a composition series
\[\begin{tikzcd}
 0=Y'_0 \arrow[r, tail] & \dots \arrow[r, tail] & Y'_{m-3} \arrow[r, tail]& Y \arrow[r, tail] & C_{m-1}
\end{tikzcd}
\] of length $m-1$.
This gives two composition series of the object $Y$ of length $n-2$ and $m-2$, respectively. By induction hypothesis, we conclude that $n-2 = m-2$ (thus $n=m$), and that these two composition series of $Y$ have the same composition factors, up to permutation and isomorphism. 
\\
Consider now the following diagram:
\\
\[
\begin{tikzcd}
0=B_0 \arrow[r, tail] & \dots \arrow[r, tail] & B_{n-2} \arrow[r, tail]             & B_{n-1} \arrow[rd, tail] &   \\
0 \arrow[r, tail]     & \dots \arrow[r, tail] & Y \arrow[ru, tail] \arrow[rd, tail] &                          & X \\
0=C_0 \arrow[r, tail] & \dots \arrow[r, tail] & C_{n-2} \arrow[r, tail]             & C_{n-1} \arrow[ru, tail] &  
\end{tikzcd}
\]
\\
By induction hypothesis, the two composition series

\[\begin{tikzcd}
 0=B_0 \arrow[r, tail] & \dots \arrow[r, tail] & B_{n-2} \arrow[r, tail]& B_{n-1}
\end{tikzcd} \]
and
\[\begin{tikzcd}
 0=Y_0 \arrow[r, tail] & \dots \arrow[r, tail] & Y \arrow[r, tail]& B_{n-1}
\end{tikzcd} 
\]
are equivalent.
\\
In the same way, the two composition series
\[ \begin{tikzcd}0=C_0 \arrow[r, tail] & \dots \arrow[r, tail] & C_{n-2} \arrow[r, tail]             & C_{n-1} \end{tikzcd}
\]
and
\[ \begin{tikzcd}0=Y'_0 \arrow[r, tail] & \dots \arrow[r, tail] & Y \arrow[r, tail]             & C_{n-1} \end{tikzcd}
\]
are equivalent.
\\\\
By the diamond axiom, 
the sets of quotients 
$\{X/B_{n-1} , B_{n-1}/Y\}$ and $\{X/C_{n-1} , C_{n-1}/Y\}$ are equal, up to isomorphism.
Using this fact and comparing the four composition series of length $n-1$ above, we conclude that the two composition series given in the beginning
have the same composition factors up to permutations and isomorphism.
\end{proof}

We provide in Section \ref{section:Artin-Wedderburn} examples of diamond categories that are not abelian categories with $\E_{all}$. 
However, not every exact category $(\A, \E)$ is diamond, or Jordan-H\"older, even if $\A$ is abelian, as the following example demonstrates.

\begin{example} \label{Counterexample1}
Consider the category $\A = \rep Q$ of representations of  the quiver
\[ \begin{tikzcd}[sep=small] & 1\arrow[dr] \arrow[dl] & \\ 2  & & 3 \end{tikzcd}\]

The Auslander-Reiten quiver of $\A$ is as follows:

\[ \begin{tikzcd}[sep= small]  S_2 \arrow[dr] & & I_3 \arrow[ll, dashed] \arrow[dr] \\ & P_1 \arrow[ur] \arrow[dr] & & I_1= S_1 \arrow[ll, dashed] \\ S_3 \arrow[ur] & & I_2 \arrow[ll, dashed] \arrow[ur]  \end{tikzcd} \]

By Theorem \ref{cube of AR-sequences}, each exact structure on $\A$ is uniquely determined by the set of Auslander-Reiten sequences which it contains. Consider the exact structure $\E$ containing the Auslander-Reiten sequences 
\begin{enumerate}
\item[(AR1)] \hspace{2.5cm}$\xymatrix{ 0 \ar[r] & S_2 \ar[r]& P_1 \ar[r] & I_3  \ar[r] & 0}$
\item[(AR2)] \hspace{2.5cm}$ \xymatrix{ 0 \ar[r] & S_3 \ar[r]& P_1 \ar[r] & I_2 \ar[r]  & 0 }$
\end{enumerate}
Then $(\A, \E)$ is not Jordan-H\"older, and it is also not a diamond category. Indeed, we have that the simples $S_2$ and $S_3$ are maximal subobjects of $P_1$, but the quotient sets $\{S_2, P_1/S_2=I_3\}$ and $\{S_3, P_1/S_3=I_2\}$ are not isomorphic.
\end{example}

\section{$\E-$Artin-Wedderburn Categories}\label{section:Artin-Wedderburn}
 We use the notion of generalised intersection to define a version of the Jacobson radical relative to an exact structure $\E.$ This allows us to show the Jordan-H\"older property for Krull-Schmidt categories under the assumption that this $\E-$radical behaves well with respect to direct sums of $\E-$simple objects, that is, the exact structure satisfies an exact analog of the Artin-Wedderburn theorem.    We then classify all such exact structures in $\rep \Lambda$ where $\Lambda$ is a Nakayama algebra and furthermore note that these are all Jordan-H\"{o}lder exact structures on $\rep \Lambda$. 
 
 Throughout this section, we assume all categories to be Krull-Schmidt categories. Recall that a is a Krull-Schmidt category an additive category, $\A$, {such that each object decomposes into a finite direct sum of indecomposable objects having local endomorphism rings} and that this decomposition is unique up to isomorphism and permutation of summands. In particular, in this case $(\A, \E_{min})$ is a Jordan-H\"older category.

\subsection{$\E-$Jacobson Radical}
Let $(\A, \E)$ be an essentially small {Krull-Schmidt} exact category.
We introduce a Jacobson radical for exact categories. 

\begin{definition}\label{radical}
Let $X \in \A$, we define the \textit{$\E-$Jacobson radical} to be the generalised intersection \[ \radical(X) := \inter_X\{(Y,f) \in \mathcal{S}_X \mid (Y,f) \in \Max(\mathcal{S}_X) \} \] 
and $\mathcal{S}_X$ is as defined previously in \ref{poset of subobjects}.
Note that, by Definition \ref{general intersection and sum}, $\radical{S} = \{0\}$ for all $\E-$simple objects $S$. 
\end{definition}

\begin{proposition}\label{properties of radical}  Consider $X, Y \in \A$ and $\begin{tikzcd}r: R \arrow[r, tail] & X.\end{tikzcd}$
\begin{enumerate}
    \item[a)] For all $(R, r) \in \radical (X)$, $\radical (\Coker(r)) = \{0\}$.
    \item[b)] For all $(Z,g) \in \mathcal{S}_X$, $Z$ is an $\E-$subobject of some $(R,r) \in \radical(X)$ if and only if $pg=0$ for all $\E-$simple quotients $p:X \twoheadrightarrow S$ of $X$.
  
\end{enumerate}
\end{proposition}
\begin{proof}
a) Let $(R, r)\in \radical (X)$ and  $(Q, q) \in \radical (X/R)$ corresponding to $Q' \rightarrowtail X$  via the Fourth $\E-$isomorphism Theorem  (Proposition \ref{4 iso}). By same result {and since $(R,r) \in \radical (X)$ we have that the} maximal $\E-$subobjects of $X$ correspond exactly to maximal $\E-$subobjects of $X / R$. Hence, as $Q$ is an $\E-$subobject of every $\E-$maximal subobject of $X/R$, we have that  $Q'$  is an $\E-$subobject of every maximal $\E-$subobject of $X$. Thus, by definition of the generalised intersection, since $R \rightarrowtail Q'$ we deduce that $R \cong Q'$ so $Q \cong Q' / R \cong 0$. 
 
 b) The claim follows from the observation that admissible epimorphisms $X\twoheadrightarrow S$ with $S$ being $\E-$simple correspond exactly to maximal $\E-$subobjects of $X$.
\end{proof}

\begin{definition} \label{esemisimple}
An object $X \in \A$ is called \textit{ $\E-$semisimple} if it can be written as a {finite} direct sum of $\E-$simple objects.
\end{definition}

We study exact categories where the $\E-$semisimple objects have nice characterisations:
\begin{definition}\label{A-W}
 An exact structure $\E$ on $\A$ is called \textit{Artin-Wedderburn} if for any object $X \in \A$ the following properties are equivalent:
\begin{enumerate}
    \item[(AW1)] Every sequence in $\E$ of the form $A \rightarrowtail X \twoheadrightarrow X/ A$ splits,
    \item[(AW2)] $X$ is $\E-$semisimple,
    \item[(AW3)] $\radical(X) = \{0\}$.
\end{enumerate} 
We say in this case that $(\A,\E)$ is an {\em $\E-$Artin-Wedderburn category.}
\end{definition}

\begin{remark} 
\begin{enumerate}
    \item[a)] {The implication (AW1) $\Rightarrow$ (AW2) always holds for Krull-Schmidt categories. Indeed, suppose $X$ is not $\E-$semisimple. Then in the decomposition of $X$ as a direct sum of indecomposables, $X \cong \bigoplus_{i =1}^n X_i$, there exists $1 \leq i \leq n $ such that $X_i$ is not $\E-$simple. Thus there exists a non-split $\E-$inflation $f: Y \rightarrowtail X$  and observe that composing $f$ with the canonical inclusion $X_i \rightarrowtail X$ results in a non-split $\E-$inflation $Y \rightarrowtail X$. 
    
   We note that without the Krull-Schmidt assumption on our categories, this implication in general does not hold, even in the abelian case. A class of counterexamples is given by the continuous spectral categories. These are Grothendieck categories where every short exact sequence splits but there are no simple objects as every object is decomposable, see \cite[Example 2.9]{PS19} for examples of such categories.}
  
  \item[b)] {The implication (AW2) $\Rightarrow$ (AW3) also always holds. Indeed, let $S_i$, $1 \leq i \leq n$ be $\E-$simple objects and $X = \bigoplus_{i=1}^n S_i$. Then observe that for all $1 \leq j \leq n $ that $\bigoplus_{i = 1, \\ i \neq j}^n S_i$ equipped with the canonical inclusion $f_j : \bigoplus_{i \neq j}S_i \rightarrowtail X$ is an $\E-$maximal subobject of $X$. Thus for every $(r:R \rightarrowtail X) \in \radical (X)$, $r$ factors through $f_j$ for all $1 \leq j \leq n$ and we deduce that $r = 0$.  }
  
\end{enumerate}
 \end{remark}

 \begin{example}
Consider the category $\A = \rep Q$ of representations of  the quiver
\[ Q: \qquad \begin{tikzcd}
1 \arrow[r, "\alpha"] & 2 & 3 \arrow[l, "\beta"'] 
\end{tikzcd}\]
We classify which exact structures $\E$ on $\A$ are Artin-Wedderburn, and when $(\A,\E)$ is a diamond or Jordan-H\"older category.
 The Auslander-Reiten quiver of $\A$ is 
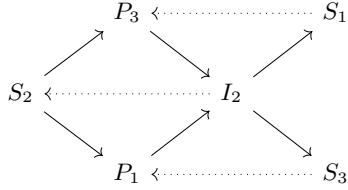
\[\begin{tikzcd}
                          & P_3 \arrow[rd] &                 & S_1 \arrow[ll, dotted] \\
S_2 \arrow[ru] \arrow[rd] &                & I_2 \arrow[ll, dotted] \arrow[ru] \arrow[rd] &                        \\
                          & P_1 \arrow[ru] &                                              & S_3 \arrow[ll, dotted]
\end{tikzcd} \] and the Auslander Reiten sequences in $\A$ are 
\begin{enumerate}
    \item[(1)] $S_2 \to P_1 \oplus P_3 \to I_2$
    \item[(2)] $P_3 \to I_2 \to S_1$
    \item[(3)]$ P_1 \to I_2 \to S_3$.
\end{enumerate}
This example has been studied in \cite[Example 4.2]{BHLR}, and $\A$ admits precisely $2^3 = 8$  exact structures $\E$ corresponding to choosing some subset $\mathcal B$ of the three Auslander-Reiten sequences in $\A$, as discussed in Theorem \ref{cube of AR-sequences}. We denote the different exact structures accordingly as 
$\E_{min}, \E(1), \E(2), \E(3),  \E(1,2),  \E(1,3),  \E(2,3), \E_{max}$, indicating the Auslander-Reiten sequences that are included.

Consider first the exact structure $\E(1)$ generated by the Auslander-Reiten sequence (1). Then the only non-split indecomposable $\E(1)$-sequence is (1) thus $P_1 \oplus P_3$ is $\E(1)$-semisimple and $(\A, \E(1))$ does not satisfy the implication (AW2) $\Rightarrow $ (AW1). The same object $P_1 \oplus P_3$ also shows that $(\A, \E(1))$ is not Jordan-H\"older (and hence not diamond) since there are non-equivalent $\E(1)-$composition series
$0 \to S_2 \to P_1 \oplus P_3$ and $0 \to P_1 \to P_1 \oplus P_3$.

Now consider the exact structure $\E(2,3)$ on $\A$ generated by the sequences (2) and (3). As in Example \ref{Counterexample1} one can see that $(\A,\E(2,3))$ is not Jordan-H\"older nor diamond. Moreover, $\mbox{rad}_{\E(2,3)}(I_2) = \{0\}$ but $I_2$ is not $\E(2,3)-$semisimple thus $(\A,\E(2,3))$ satisfies neither the implication  (AW3) $\Rightarrow$ (AW1) nor (AW3)$\Rightarrow$ (AW2).

One may verify that all other exact structures $\E$ on $\A$ are Artin-Wedderburn, and also satisfy the diamond and Jordan-H\"older property, but only $(\A, \E_{max})$ is an AIS-category.
We conclude that six of the eight exact structures are Jordan-H\"older, and in this example, the conditions being $\E-$Artin-Wedderburn, diamond and Jordan-H\"older are equivalent.
\end{example}
\bigskip

A further example of $\E-$Artin-Wedderburn categories is provided by the split exact structure:

\begin{lemma}\label{Emin is W-A}
$\A$ is an $\E_{min}$-Artin-Wedderburn category.
\end{lemma}
\begin{proof}
For the exact structure $\E=\E_{min}$, we have that the admissible monics are precisely the sections, and the $\E-$simple objects are the indecomposables.  Every object in $\A$ is thus $\E-$semisimple, and we clearly have the equivalence of (AW1) and (AW2). Since every $X$ is $\E-$semisimple, the implication (AW3) $\Longrightarrow$ (AW2) is always true. 
\end{proof}

As we have noted, for Krull-Schmidt categories, $(\A, \E_{min})$ is a Jordan-H\"older category. 
The following result further studies the relationship between Krull-Schmidt categories and the Jordan-H\"older property.
\color{black}

\begin{theorem}\label{thm:AW->JH}
Let $(\A, \E)$ be an $\E$-Artin-Wedderburn category. Then $(\A, \E)$ is a Jordan-H\"{o}lder exact category.
\end{theorem}

\begin{proof}
We show that $(\A, \E)$ satisfies the Diamond Axiom \ref{Diamond}. {For that purpose}, let 
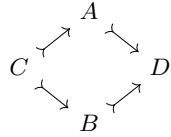
\[ \begin{tikzcd}[sep=small]
& A \arrow[dr, tail] & \\ C \arrow[ur, tail] \arrow[tail, dr] & & D \\ & B \arrow[ur, tail] & 
\end{tikzcd}\] be a commutative diagram in $(\A, \E)$ with $D/ A$ and $D/ B$ being $\E$-simple and $C \in \inter_{D}(A, B)$. By the Fourth $\E-$Isomorphism Theorem  (Proposition \ref{4 iso}), there is a commutative diagram 
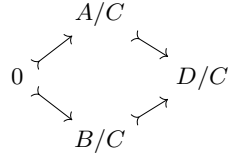
\[ \begin{tikzcd}[sep=small]
& A/ C \arrow[dr, tail] & \\ 0 \arrow[ur, tail] \arrow[tail, dr] & & D/C \\ & B/C \arrow[ur, tail] & 
\end{tikzcd}\] with $(D/ C) / (A/C) \cong D /A$ and $(D/C) / (B/C) \cong D/ B$ being $\E-$simple and $\inter_{D/C}(A/C, B/C) = \{0\}$. Thus, it is enough to consider diagrams of the form 
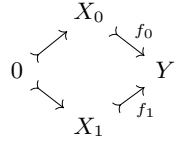
\[\begin{tikzcd}[sep=small]
& X_0 \arrow[dr, tail, "f_0"] & \\ 0 \arrow[ur, tail] \arrow[tail, dr] & & Y \\ & X_1 \arrow[ur, tail, "f_1"'] & 
\end{tikzcd}\] with $ Y / X_i$ being $\E-$simple for $i=0,1$ and $\inter_Y(X_0, X_1) = \{0\}$. 

We must show that the $X_i$ are $\E-$simple and that the {sets $\{X_0, Y / X_0\}, \{X_1, Y / X_1\}$ are equal up to permuation and isomorphism of their elements.}

If $(X_0,f_1)$ and $(X_1,f_1)$ are isomorphic as $\E-$subobjects of $Y$ {it follows that $X_0 \cong X_1$ is $\E-$simple since $\inter_Y(X_0, X_1) = \{0\}$. So we may assume that this is not the case.} Observe that the  $(X_i,f_i)$ are both maximal $\E$-subobjects of $Y$. It follows that $\radical(Y) \subset \inter_Y ( (X_0,f_0), (X_1, f_1)) = \{0\}$. Since $(\A, \E)$ is $\E$-Artin-Wedderburn, the short exact sequences 
$\begin{tikzcd} X_i \arrow[r, tail, "f_i"] & Y \arrow[r, two heads] & Y / X_i \end{tikzcd}$ 
both split and $Y$ is $\E$-semisimple. Thus $X_0 \oplus Y / X_0 \cong Y \cong X_1 \oplus Y / X_1$  and $X_0 \cong \bigoplus_{j=0}^n S_j$ and $X_1 \cong \bigoplus_{j=0}^m T_j$ with the $S_j$ and $T_j$ being $\E$-simple. As $(\A, \E)$ is Krull-Schmidt, $n=m$ and {the sets $\{ S_0, \dots, S_n, Y / X_0\}, \{T_0, \dots, T_n, Y / X_1\}$ consist of the same objects, up to permutation and isomorphism.} Without loss of generality, we may suppose that $S_0 \cong Y / X_1$ and $T_0 \cong Y / X_0$. Now $\oplus_{j=1}^n S_j \rightarrowtail X_i$, but since $\inter_{Y}(X_0, X_1) = \{0\}$ we conclude that $n=0$ and that the $X_i$ are $\E-$simple and we are done. \end{proof}

\subsection{The Artin-Wedderburn exact structures for Nakayama algebras}
We characterise all Artin-Wedderburn exact structures for any Nakayama algebra $\Lambda.$ It turns out they are exactly the Jordan-H\"older exact categories for {$\mod \Lambda,$ the category of finitely generated left $\Lambda-$modules.}

A finite-dimensional algebra $\Lambda$ is called \emph{Nakayama} if every indecomposable right and left projective $\Lambda$-module is uniserial. The representation theory of Nakayama algebras is well-known ( see e.g. \cite[Chapter V]{ASS} or \cite[Section VI.2]{ARS}), we recall some details here:

The indecomposable $\Lambda-$modules are all uniserial, thus determined by the list of its composition factors from top to socle, which can be represented by a word $w$ in the vertices of the quiver of $\Lambda$. Denote the module corresponding to a word $w$ by $[w]$. Equivalently, indecomposable $\Lambda-$modules are parametrized by the non-zero paths in the quiver $Q$ of $\Lambda.$ 

If we label the vertices of the path in $Q$ corresponding to the indecomposable module $[w]$ as 
\[ c \to c+1 \to  \dots \to d-1 \to d \]
then we denote the module $[w]$ also by $[w]= [c,d]$. In this case, the Auslander-Reiten quiver of $\Lambda$ contains a subquiver of the form described in Figure \ref{fig:ARquiver}
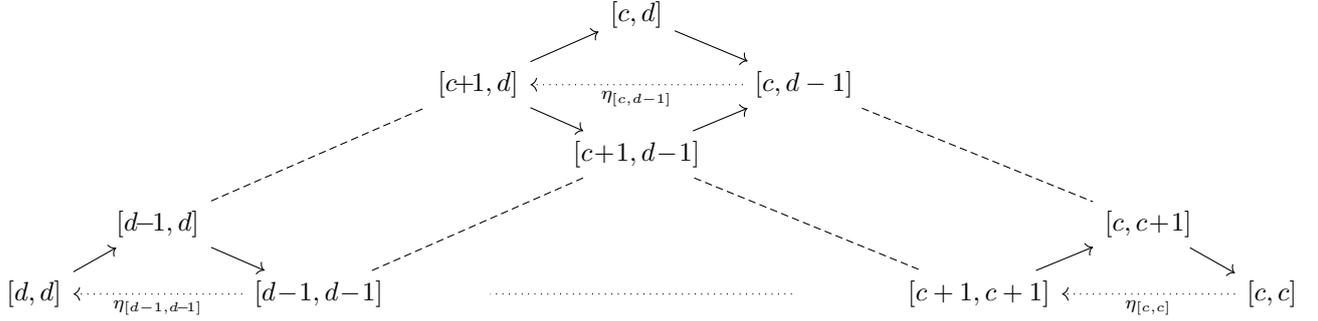
\begin{figure}
    \centering
\[
\begin{tikzcd}[sep = small]
 & & & & {[c,d]} \arrow[rd] &    & & &  \\
 & & & {\small [c\!\!+\!\!1,d]} \arrow[rd] \arrow[ru] & & {[c,d-1]} \arrow[rrdd, no head, dashed] \arrow[ll, dotted, "\eta_{[c,d-1]}"] & & & \\
 & & & & {[c\!+\!1,d\!-\!1]} \arrow[rrdd, no head, dashed] \arrow[ru] & & &  & \\
 & {[d\!\!-\!\!1,d]} \arrow[rd] \arrow[rruu, no head, dashed] & & & & &  & {[c,c\!+\!1]} \arrow[rd] &                              \\
{[d,d]} \arrow[ru] &      & {[d\!-\!1,d\!-\!1]} \arrow[ll, dotted, "\eta_{[d-1,d\!-\!1]}"] \arrow[rruu, no head, dashed] & {}  & {} & {} \arrow[ll, no head, dotted]   & {[c+1,c+1]} \arrow[ru] &    & {[c,c]} \arrow[ll, dotted,  "\eta_{[c,c]}"] 
\end{tikzcd} \] 
    \caption{Part of the Auslander-Reiten quiver of $\rep \Lambda$ containing the module $[c,d]$ and all of its simple composition factors}
    \label{fig:ARquiver}
\end{figure}
where we label the Auslander-Reiten sequences $\eta_{[c,d-1]}$ in $\A = \mod \Lambda$  by the module $[c,d-1]$ where the sequence ends; the sequence starts in the Auslander-Reiten translate $\tau [c,d-1] = [c+1,d]$.
For indecomposables $[w]$ and $[w']$, the space \[\Ext_\Lambda^1([w],[w'])\] is at most one-dimensional, and a basis can be given by the following non-split short exact sequences:
If $[ww']$ is indecomposable, then a basis is given by
\[ \eta_{w,w'}: 0 \longrightarrow [w] \longrightarrow [ww'] \longrightarrow [w'] \longrightarrow  0.
\]
If $w=uv$ and $w'=vt$ such that $[uvt]$ is indecomposable, then a basis is given by
\[ \eta_{w,w'}: 0 \longrightarrow [w] \longrightarrow [v] \oplus [uvt] \longrightarrow [w'] \longrightarrow  0.
\]
We refer to $[ww']$ respectively $ [uvt]$ as the {\em top module} in the extension $\eta_{w,w'}.$
Thus for the Auslander-Reiten sequence $\eta_{[c,d-1]}$, the top module is $[c,d]$.
The description of the indecomposables and the Auslander-Reiten sequences in $\A$ 
can be obtained from \cite{BR} for the more general case of string algebras, and the basis  for the $\Ext^1$-spaces is given in \cite{BDMTY} for gentle algebras.

Our first step is to give a more precise description of an exact structure on $\A$ using the Auslander-Reiten sequences it contains.

\begin{theorem}\label{EB description}
 Let $\mathcal{B}$ be a set of Auslander-Reiten sequences in $\A$ and $\E=\E(\mathcal{B})$ be the corresponding exact structure on $\A$. Then the short exact sequence $\eta_{w,w'} \in \E$ if and only if the Auslander-Reiten sequence $\eta_{[u]}$ belongs to $\mathcal{B}$ whenever there is a non-zero morphism from $[w]$ to $\tau[u]$ and from $[u]$ to $[w'].$
\end{theorem}
\begin{proof}
Necessity follows directly from axioms (A2) and (A2)$^{op}$:  One can easily verify that forming push-outs and pull-backs of the given  exact sequence $\eta_{w,w'}$   along the morphisms from $[w]$ to $\tau[u]$ and from $[u]$ to $[w']$ yields the desired Auslander-Reiten sequences, which thus belong to $\E$.

Sufficiency follows from the fact that exact structures $\E$ on $\A$ correspond to {\em closed} subfunctors of the bifunctor $\Ext^1(-,-)$ on $\A,$ see \cite{DRSS}. Auslander-Reiten theory implies that the socle of $\Ext^1(-,-)$ is given by the Auslander-Reiten sequences, and a closed subfunctor $\E= \E(\mathcal{B})$ is uniquely determined by its socle $\B,$ see \cite{BBH}.
We show in \cite{BBH} that $\E= \E(\B)$ is the maximal subfunctor of $\Ext^1(-,-)$ whose socle is $\B,$ therefore the sequence  $\eta_{w,w'}$ (which induces the socle elements in $\B$ as we showed above when discussing necessity) must belong to $\E= \E(\B)$. 
Here we indicate how to verify this directly from the axioms and leave the details to the reader: 

Consider first the case where $[w]= \tau [u]$ and there is an arrow in the Auslander-Reiten quiver from $[u]$ to $[w']$:
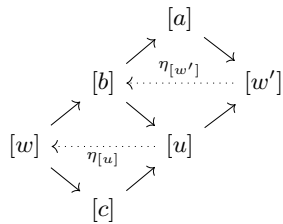
\[
\begin{tikzcd}[sep = small]
 & &  {[a]} \arrow[rd] &   \\
 &  { [b]} \arrow[rd] \arrow[ru] & & {[w']}  \arrow[ll, dotted, "\eta_{[w']}"']  \\
 {[w]} \arrow[rd] \arrow[ru] & & {[u]} \arrow[ll, dotted, "\eta_{[u]}"] \arrow[ru]   & \\
 & {[c]} \arrow[ru] &  &      \\
\end{tikzcd} \] 

{ By assumption, the Auslander-Reiten sequence $\eta_{[u]}$ belongs to $\B$ since there is an irreducible morphism from $[u]$ to $[w']$. Moreover, since the identity is a non-zero morphism, the Auslander-Reiten sequence $\eta_{[w']}$ also belongs to $\B$. We wish to apply axiom (A1) of an exact structure to this situation, however the monos from $\eta_{[u]}$ and $\eta_{[w']}$ cannot be composed directly, only when considering the direct sum of  the split exact sequence $(id_{[c]},0)$  with $\eta_{[w']}$ this becomes possible. 
It turns out that the composition of the mono from $\eta_{[u]}$ with the mono of the short exact sequence $\eta_{[w']} \oplus   (id_{[c]},0)$ yields the mono of the short exact sequence  $\eta_{w,w'} \oplus   (id_{[c]},0)$, which belongs to $\E$ by axiom (A1). Then \cite[Corollary 2.18]{Bu} shows that  $\eta_{w,w'} \in \E.$ 
}
To finish the proof, proceed by induction along paths from $[w]$ to  $\tau [w']$ and from $\tau^{-1}[w]$ to  $[w']$.
\end{proof}

\begin{remark}
As $\Lambda$ is Nakayama, the poset of submodules of an indecomposable $[c,d]$  is totally ordered. In particular, for any exact structure $\E$ on $\A$, the poset of proper $\E-$subobjects $\mathcal{S}_{[c,d]}$ is also totally ordered. Hence all indecomposable non $\E-$simple objects have a unique maximal $\E-$subobject. Moreover,  all ($\E-$)subobjects of $[c,d]$ are of the form $[x,d]$ for some  $c \le x \le d$, whereas all quotients are of the form $[c,y]$ for some  $c \le y \le d$, see Figure 1.
\end{remark}

Now we may classify all Artin-Wedderburn exact structures on $\A=\mod \Lambda$ when $\Lambda$ is Nakayama.

\begin{theorem}\label{EAW example}
 Let $\mathcal{B}$ be a set of Auslander-Reiten sequences in $\A=\mod \Lambda$ and $\E=\E(\mathcal{B})$ be the corresponding exact structure on $\A$. Then $\E$ is Artin-Wedderburn if and only if for all Auslander-Reiten sequences $\eta_{[w]} \in \mathcal{B}$ the top module of this sequence is not $\E$-simple.
\end{theorem}

\begin{proof}
We use the notation from Figure \ref{fig:ARquiver}.
To simplify the presentation of the proof, we introduce phantom zero objects $[x,y]=0$ whenever $x > y$. In this notation, all Auslander-Reiten sequences $$\begin{tikzcd} \eta_{[c,d-1]}:  [c+1,d ] \arrow[r, tail] & {[c, d] \oplus [c+1, d-1]} \arrow[r, two heads] & {[c, d-1]} \end{tikzcd}$$ have two middle terms, with top module $[c, d]$, and where $[c+1, d-1]$ denotes the zero object when $c+1 > d-1.$

We first suppose that there exists an Auslander-Reiten sequence $\eta_{[c,d-1]}$ in $\mathcal B$
 such that the top module $[c, d]$ is $\E-$simple, and we show that this implies $\E$ being not Artin-Wedderburn. 
Let $y \le d-1 $ be such that $[c, y]$ is $\E$-simple and $\eta_{[c,j]} \in \mathcal{B}$ for all $j \in (y,d-1]$. Such a $y$ always exists. Indeed, if $[c, d-1]$ is $\E$-simple then we take $y = d-1$. Else, let $[c, y]$ be an $\E$-simple factor module of $[c, d-1]$, then $y$ satisfies the required conditions by Theorem \ref{EB description}. 

Now, let $x \geq c$ be maximal such that 
there is an indecomposable non-split short exact sequence in $\E$ of the form

\begin{equation}
 \begin{tikzcd} {[x, d]} \arrow[r, tail] & {[c, d] \oplus [x, y]} \arrow[r, two heads] & {[c, y].} \end{tikzcd}
\end{equation} 
Note that $[x, y] \not \cong 0$ by the assumption that $[c,d]$ is $\E$-simple. If $[x, y]$ is $\E$-simple then this sequence shows that the implication (AW2) $\Rightarrow$(AW1) does not hold. Suppose that $[x, y]$ is not $\E$-simple and let $[w, y]$ be its unique maximal $\E-$subobject, note that $w>x$. 
Thus $$\radical ( [c, d] \oplus [x, y]) \subseteq \inter_{[c, d] \oplus [x, y]}([c, d] \oplus [w, y], [x,d] ).$$ 
Observe that the $\E$-subobjects of $[x,d]$ are of the form $[i,d]$ with $i \in (x, d]$ and the only possible indecomposable $\E$-subobjects of $[c, d] \oplus [w, y]$ are of the form $[j, y]$ with $j \in (w, y]$ or $[w,d]$. 
We deduce that, if $\radical ( [c, d] \oplus [x, y]) \neq \{0\}$ then $[w,d]$ is an $\E$-subobject of $[c, d] \oplus [w, y]$. But this is a contradiction to the maximality of $x$, thus the implication (AW3) $\Rightarrow$ (AW1) does not hold.
\medskip

For the converse, consider a non-split $\E-$sequence. Since $\Ext^1(-,-)$ is an additive bifunctor, it suffices to consider short exact sequences with indecomposable end terms, which are for Nakayama algebras of the form 
\[ \begin{tikzcd} {[c, d]} \arrow[r, tail] & {[a, d] \oplus [b, c]} \arrow[r, two heads] & {[a, b]} \end{tikzcd} \] where $[b,c]$ may denote the zero object. Note that the inequalities $a \leq c-1 \leq b \leq d-1$ must hold. By assumption and Theorem \ref{EB description};
$[a, d]$ is not $\E-$simple. Thus, since $\A$ is Krull-Schmidt, $[a, d] \oplus [b, c]$ is not $\E-$semisimple. This shows that the implication (AW2) $\Rightarrow$ (AW1) holds. It remains to show that $\radical ({[a, d] \oplus [b, c]}) \neq \{0\}$. 
 Observe that
 \[ \mathcal{S}_{[a, d] \oplus [b, c]} = \Big\{[c, d], \,[i, d] \oplus [c, b], \, [a, d] \oplus [j,b ] \mbox{ with } [i,d ] \in \mathcal{S}_{[a, d]}, \, [j, b] \in \mathcal{S}_{[c,b ]} \Big\} \] 
 Let $[x, d] \rightarrowtail [a,d]$ be the unique maximal $\E-$subobject which exists by assumption that the top module $[a,d]$ is not $\E-$simple, and let $[y,b] \rightarrowtail [c, b]$ be the unique maximal $\E-$subobject of $[c, b]$ if it exists or the identity if not. Then
 \[ \Max(\mathcal{S}_{[a, d] \oplus [c, b]}) \subseteq \Big\{ [c, d], [x, d] \oplus [c,b], \, [a,d ] \oplus [y,b] \Big\}. \]

First suppose that $x \geq c$. Then, as $a <c$ by Theorem \ref{EB description}, $[x,d] \rightarrowtail [c,d]$ and 
 \[ \radical \Big([a, d] \oplus [c, b] \Big) \supseteq \inter_{[a, d] \oplus [c, b]} \Big([c, d], [x, d] \oplus [c,b], \, [a,d ] \oplus [y,b] \Big) = \Big\{ [x,d] \Big\} \neq 0. \]

Now suppose that $x <c$.  Then $b>x$ as $c-1 \leq b$.  Now
\[\radical \Big([a, d] \oplus [c,b] \Big) \supseteq \inter_{[a, d] \oplus [c, b]} \Big([c, d], [x, d] \oplus [c,b], \, [a,d ] \oplus [y,b] \Big) = \Big\{ [x,d] \oplus [y,b] \Big\} \neq 0 \] and we are done.
\end{proof}
\color{black}

Note that Enomoto studies in \cite{E19}  the Jordan-H\"older property for torsion-free classes in the module category of a Nakayama algebra endowed with the maximal exact structure.  We investigate now when $\A=\mod \Lambda$ with any exact structure $\E$ is Jordan-H\"older:

\begin{theorem}\label{Nakayama}
Let $\Lambda$ be a Nakayama algebra, and denote $\A = \mod \Lambda.$
Then an exact category $(\A,\E)$ is $\E-$Artin-Wedderburn precisely when it is Jordan-H\"older.
\end{theorem}
\begin{proof}
The $\E-$Artin-Wedderburn categories are  Jordan-H\"older by Theorem \ref{thm:AW->JH}.
Conversely, assume that $(\A,\E)= (\A,\E(\B))$ is  Jordan-H\"older.
By  \cite[Theorem 4.13]{E19}, we know that the number $s$ of $\E-$simple objects equals the number $p$ of $\E-$projective indecomposable objects. 
Every non-$\E-$projective indecomposable admits an Auslander-Reiten sequence in $\B$, therefore 
\[ s = p = |ind (\A)| - |\B| \]
where $ind (\A)$ denotes the (isoclasses of) indecomposables in $\A.$ We conclude
\[ |ind (\A)| =  |\B| +s, \]
and parametrise the set of indecomposables by the $\E-$simples together with the top module for every Auslander-Reiten sequence. 
Clearly this top module cannot be $\E-$simple in this case,  {thus by Theorem \ref{EAW example}, $(\A,\E)$ is $\E-$Artin-Wedderburn. }
\end{proof}

\section{The length function }

In this section, we consider a Jordan-H\"older exact category $(\A, \E)$ and we study the $\E-$Jordan-H\"older length function  $l_{\E}$ that the $\E-$Jordan-H\"older theorem  allows us to define over the set $Obj\mathcal{A}$ of isomorphism classes of objects. Throughout, $(\A,\E)$ denotes an $\E-$finite essentially small Jordan-H\"older exact category.
To simplify notation, we will not distinguish here between the isomorphism class $[X]$ of an object $X$ of $\A$ and the object $X$.


\begin{definition}\label{length} We define the $\E-$Jordan-H\"older length $l_{\mathcal{E}}(X)$ of 
an object $X$ in $\A$
as the length of an $\E-$composition series of $X$. That is $l_{\mathcal{E}}(X)=n$ if and only if there exists an $\E-$composition series 
\[
\begin{tikzcd}0=X_0 \arrow[r, tail] & X_1 \arrow[r, tail] & \dots \arrow[r, tail]            & X_{n-1} \arrow[r, tail] & X_n=X. \end{tikzcd}\] 
We say in this case that $X$ is $\E-$finite. 
If no such bound exists, we say that $X$ is $\E-$infinite.
Clearly, isomorphic objects have the same length, and therefore this definition gives rise to a length function $l_{\mathcal{E}}: Obj\mathcal{A}\rightarrow \mN \cup \{ \infty\}$ defined on isomorphism classes.
\end{definition}

Now we prove some corollaries of the $\E-$Jordan-H\"older theorem:

\begin{cor}\label{s.e.s length} Let \[X \; \rightarrowtail Z \twoheadrightarrow  Y \] be an admissible short exact sequence of finite length objects. Then

$$l_{\mathcal{E}}(Z) = l_{\mathcal{E}}(X) + l_{\mathcal{E}}(Y).$$

\end{cor}
\begin{proof}
We know that $X$ is a subobject of $Z$ and that $Y\cong Z/X$. We consider the following composition series of $X$ and $Y$
\[\begin{tikzcd}0=X_0 \arrow[r, tail] & X_1 \arrow[r, tail] & \dots \arrow[r, tail]            & X_{n-1} \arrow[r, tail] & X_n=X \end{tikzcd}\]
\[\begin{tikzcd}0=Z_0/X \arrow[r, tail] & Z_1/X \arrow[r, tail] & \dots \arrow[r, tail]            & Z_{l-1}/X \arrow[r, tail] & Z_l/X\cong Y \end{tikzcd}\]
where we us the fourth $\E-$isomorphism theorem (Proposition \ref{4 iso}) to obtain the particular structure for the composition series of $Y$. 
Since  \[(Z_{i+1}/X)/(Z_{i}/X)\cong (Z_{i+1}/Z_{i})\]
by \cite[Lemma 3.5]{Bu}, the following is a composition series of $Z$:
\[\begin{tikzcd}[row sep=tiny]0=X_0 \arrow[r, tail] & X_1 \arrow[r, tail] & \dots \arrow[r, tail]            & X_{n-1} \arrow[r, tail] & X_n=X=Z_0 \arrow[r,tail, "i"] & {} \\ & \arrow[r, tail,"i"] & Z_1 \arrow[r, tail] & \dots \arrow[r, tail]            & Z_{l-1} \arrow[r, tail] & Z_l=Z\end{tikzcd}\]
Thus $$l_{\mathcal{E}}(Z)= n + l = l_{\mathcal{E}}(X) + l_{\mathcal{E}}(Y)$$
\end{proof}

We show now that the function $l_{\mathcal{E}}$ is a \emph{length function} in the sense of \cite{Kr07}:
\begin{definition}\label{Ameasure}  A {\em measure for a poset $\cS$} is a morphism of posets $\mu :\cS\rightarrow P$ where $(\mathcal{P}, \leq)$ is a totally ordered set.
A measure $\mu$ is called a \emph{length function} when $\mathcal{P} =\mathbb{N}$ with the natural order. 
\end{definition}

\begin{theorem} \label{lengththm} The function $l_{\mathcal{E}}$ of an $\E-$finite Jordan-H\"older exact category $(\A, \E)$ is a \emph{length function} for the poset $Obj\mathcal{A}$.
\end{theorem}
\begin{proof}The function $l_{\mathcal{E}}: Obj\mathcal{A}\rightarrow \mN $ is defined on the set $Obj\A$, which is  partially ordered by the $\E-$subset relation $X\subset_\E Y$, see \cite[Proposition 6.11]{BHLR}. Moreover,
consider $X$ and $Y$ in $Obj\A$ with $X\subset_\E Y$. Then by Corollary \ref{s.e.s length} we have
\[ l_{\mathcal{E}}(X)\leq l_{\mathcal{E}}(Y),\] 
so $l_{\E}$ is a morphism of posets.
\end{proof}

As a consequence of the previous result, 
 an $\E-$finite object is an object with $\E-$finite length.

\begin{lemma}\label{c.s}Let $(\A, \E)$ be an exact category. An $\E-$Artinian and $\E-$Noetherian object $X$ of $(\A, \E)$ admit an $\E-$composition series.
\end{lemma}
\begin{proof}Let $X$ be an $\E-$Artinian and $\E-$Noetherian object. Using the artinian hypothesis, one can construct a sequence of strict $\E-$subobjets with $\E-$simple quotients :
\[
\begin{tikzcd}0=X_0 \arrow[r, tail] & X_1 \arrow[r, tail] & \dots   \end{tikzcd}\]
Since $X$ is noetherian too, this sequence became stationary and end with $X$ at some point. Finally, this sequence give an $\E-$composition series in the sense of \ref{composition series}.
\end{proof}
The following results improves and uses \cite[Lemma 6.5]{BHLR} and \ref{c.s}:
\begin{theorem}\label{hopkinslevitzki} Let $(\A, \E)$ be a Jordan-H\"older exact category. An object $X$ of $(\A, \E)$ is $\E-$Artinian and $\E-$Noetherian if and only if it has an $\E-$finite length.
\end{theorem}

\begin{proof}
For an $\E-$finite object $X$ of length $l_{\E}(X)=n \in \mathbb{N}$, the composition series is of length $n$. Thus any increasing or decreasing sequence of $\E-$subobjects of $X$ must become stationary and $X$ is $\E-$Artinian and $\E-$Noetherian.\\ Conversely, let $X$ be an $\E-$Artinian and $\E-$Noetherian object, then $X$ admits an $\E-$composition series by \ref{c.s}. Since $\E$ satisfies the Jordan-H\"older property, all composition series ending with $X$ have the same finite length, so $X$ is $\E-$finite.
\end{proof}

\begin{remark}
Note that a length function for exact categories in general was studied in  \cite[Section 6]{BHLR}. The notion there was defined as maximum over all lengths of an $\E-$composition series; in the case of an $\E-$Jordan-H\"older category all composition series of an object have the same length, so the definition we use here is compatible with the one from \cite{BHLR}. 
\end{remark}

\begin{definition} We denote by $(Ex({\A}), \subseteq)$ the poset of exact structures $\E$ on $\A$, where the partial order is given by containment $\E' \subseteq \E$.
This \emph{containment} partial order is studied in \cite[Section 4]{BHLR}.
\end{definition}

We conclude by noting that, similarly to \cite[Lemma 8.1]{BHLR}, the $\E-$Jordan H\"older length function can only decrease under reduction of exact structures:

\begin{proposition} \label{lengthreduction}
 If $\E$ and $\E'$ are exact structures on $\A$ such that $\E' \subseteq \E$, then $l_{\E'}(X) \leq l_{\E}(X)$ for all objects $X$ in $\A$.
\end{proposition}
\begin{proof}Let us consider an $\E'-$composition series of ending by $X$
\[ 0=X_0 \; \imono{i_1} X_1 \;\imono{i_2} \cdots \; \imono{i_{n-1}} \;  X_{n-1}\;\imono{i_n}X=X_n \]
where $l_{\E'}(X)=n$. Since $\E'\subseteq \E$, all these pairs $(i_j, d_j)$  will also be in $\E$. So the
$\E'-$composition series is also an $\E-$composition series and therefore by definition $l_{\E}(X)\ge n$.
\end{proof}



\bigskip

\noindent{Thomas Br\"ustle, Souheila Hassoun \\D\'{e}partment de math\'{e}matiques \\  Universit\'{e} de Sherbrooke \\  Sherbrooke, Qu\'{e}bec, J1K 2R1 \\  Canada }\\
Thomas.Brustle@usherbrooke.ca \\ Souheila.Hassoun@usherbrooke.ca 

\bigskip 

\noindent{Aran Tattar \\ School of Mathematics \\ University of Leicester \\  Leicester, LE1 7RH \\  UK }\\
ast20@le.ac.uk


\begin{thebibliography}{ABCDE99}

\bibitem[ASS06]{ASS}
  I. Assem, D. Simson, A. Skowro\'nski,
  \emph{Elements of the representation theory of associative algebras Vol. 1}, London Mathematical Society Student Texts, 65,
Cambridge University Press, Cambridge, 2006.

  \bibitem[ARS95]{ARS}
  M. Auslander, I. Reiten, S.O. Smal\o,
  \emph{Representation theory of Artin algebras},
  Cambridge Studies in Advanced Mathematics, 36. Cambridge University Press, Cambridge, 1995.


\bibitem[BBGH20]{BBH} R.-L. Baillargeon, Th. Br\"ustle, M. Gorsky, S. Hassoun, {\em On the lattice of weakly exact structures,} arXiv:2009.10024.

\bibitem[Ba06]{Bau}B. Baumslag, {\em A simple way of proving the Jordan-Hölder-Schreier theorem}, American Mathematical Monthly, 113 (2006) Issue 10, 933--935.


\bibitem[BS88]{BS} C. Bennett, R. Sharpley, {\em Interpolation of operators,}  Pure and Applied Mathematics,
Volume 129,
Pages 1--469 (1988).


\bibitem[BG16]{BeGr}A. Berenstein, J. Greenstein, {\em Primitively generated Hall algebras,} Pacific J. Math. 281 (2016), no. 2, 287--331.

\bibitem[BDMTY17]{BDMTY} Th. Br\"ustle, G. Douville, K. Mousavand, H. Thomas, E. Y\i ld\i r\i m, {\em On the combinatorics of gentle algebras,} to appear in Canadian Journal of Mathematics, arXiv:1707.07665, 2017.
	
\bibitem[BHLR20]{BHLR} Th. Br\"ustle, S. Hassoun, D. Langford, S. Roy, {\em Reduction of exact structures,}
J. Pure Appl. Algebra 224 (2020), no. 4, 29 pp.

\bibitem[BrHi00]{BrHi} Th. Br\"ustle, L. Hille, {\em Matrices over upper triangular bimodules and $\Delta-$filtered modules over quasi-hereditary algebras,} Colloq. Math. 83 (2000), no. 2, 295--303.

\bibitem[B\"u10]{Bu} T. Bühler, \emph{Exact categories.} Expo. Math. 28 (2010), no. 1, 1--69.

\bibitem[B\"u11]{Bu11} T. Bühler, {\em On the algebraic foundations of bounded cohomology}, Mem. Amer. Math. Soc. 214 (2011), no. 1006.

\bibitem[BR87]{BR} M. C. R. Butler, C. M. Ringel, 
\emph{Auslander-Reiten sequences with few middle terms}, 
Comm. in Algebra. 15 (1987), 145--179.

\bibitem[Ch10]{Ch} H. Chen, {\em Harder-Narasimhan categories}, J. Pure Appl. Algebra 214 (2010), no. 2, 187--200. 
\bibitem[DRSS99]{DRSS} P. Dräxler, I. Reiten, S.O. Smal\o, \O. Solberg, {\em Exact Categories and Vector Space Categories,} Transactions of the American Mathematical Society, vol.351, no.2, 1999.

\bibitem[E16]{E16} H. Enomoto, \emph{Classifying exact categories via Wakamatsu tilting,} Journal of Algebra, Volume 485, 1 September 2017, pages 1-44, 2016.


\bibitem[E17]{E17} H. Enomoto, {\em Classifications of exact structures and Cohen–Macaulay-finite algebras,} Advances in Mathematics, Volume 335, 7 September 2018, pp.838-877, 2017.

\bibitem[E18]{E18} H. Enomoto, \emph{Relations for Grothendieck groups and representation-finiteness,} Journal of Algebra, Volume 539, December 2019, pages 152-176, 2018.


\bibitem[E.19]{E19} H. Enomoto, \emph{The Jordan-H\"older property and Grothendieck monoids of exact categories,} arXiv:1908.05446, 2020

\bibitem[E20]{E20}H. Enomoto, \emph{Schur's lemma for exact categories implies abelian,} arXiv: 2002.09241, 2020.


\bibitem[FG20]{FG} X. Fang, M.Gorsky, \emph{Exact structures and degeneration of Hall algebras}, arXiv: 2005.12130, 2020.

\bibitem[F66]{Freyd}P. Freyd, \emph{Representations in abelian categories}, Proc. Conf. Categorical Algebra (La Jolla, Calif., 1965), Springer, New York, 1966, pp. 95–120. 

\bibitem[GR92]{GR} P. Gabriel, A.V. Ro\u{\i}ter, \emph{Representations of Finite-dimensional Algebras,} in: Algebra, VIII, Encyclopaedia Mathematical Sciences, vol. 73, Springer, Berlin, 1992 (with a chapter by B. Keller), pp. 1--177, 2020. 


\bibitem[HR19]{HR} S. Hassoun, S. Roy, \emph{Admissible intersection and sum property}, arXiv: 1906.03246.

\bibitem[HSW20]{HSW} S. Hassoun, A. Shah, S-A. Wegner,
\emph{Examples and non-examples of integral categories }, arXiv:2005.11309, 2020.

\bibitem[J17]{J} P. Jiao, {\em Auslander's defect formula and a commutative triangle in an exact category}, arXiv:1707.01646.

\bibitem[Ke69]{Kelly} G.M. Kelly, {\em 
Monomorphisms, epimorphisms, and pull-backs}, J. Austral. Math. Soc.,
9: 124--142, 1969.


\bibitem[KKO14]{KKO} S. Koenig, J. K\"ulshammer, S. Ovsienko, {\em  Quasi-hereditary algebras, exact Borel subalgebras, $A_\infty$-categories and boxes}, Adv. Math. 262 (2014), 546--592.


\bibitem[Kr07]{Kr07} H. Krause,\emph{ An axiomatic characterization of the Gabriel-Roiter measure}. Bull. Lond. Math. Soc. 39 (2007), no. 4, 550--558.


\bibitem[Ma86]{Ma} L. Maligranda,
{\em Interpolation between sum and intersection of Banach spaces},
Journal of Approximation Theory 47 (1986),  no. 1, 42--53.


\bibitem[Pa70]{Pa}
B. Pareigis, {\em Categories and functors}, Pure and Applied Mathematics, Vol. 39 Academic Press, New York-London 1970, 268 pp.

\bibitem[P19]{P} P. Patak, {\em Jordan-H\"older with uniqueness for semimodular semilattices}, arXiv:1908.09912.

\bibitem[Po73]{pop} N. Popescu, {\em Abelian categories with applications to rings and modules}, Vol. 3. Academic Press, 1973.

\bibitem[PS19]{PS19} L. Positselski, J. Stovicek {\em Topologically semisimple and topologically perfect topological rings},  arXiv:1909.12203 (2019)


\bibitem[Qu73]{Qu}
D. Quillen, \emph{Higher algebraic {$K$}-theory. {I}}, Algebraic
  $K$-theory, I: Higher $K$-theories (Proc. Conf., Battelle Memorial Inst.,
  Seattle, Wash., 1972), Springer, Berlin, 1973, pp.~85--147. Lecture Notes in
  Math., Vol. 341.
  
  
\bibitem[RW77]{RW77} F. Richman, E.A. Walker, \emph{Ext  in pre-Abelian  categories},  Pacific  J. Math.  71(1977),  521--535.
  
\bibitem[Ru01]{Ru01} W. Rump, \emph{Almost abelian categories}, Cahiers Topologie G\'eom. Diff\'erentielle Cat\'eg. 42 (2001), no. 3, 163--225. 

\bibitem[Sa19]{Sa} V. Santiago, {\em Stratifying systems for exact categories,} Glasg. Math. J. 61 (2019), no. 3, 501--521.
\color{black}

\bibitem[SW11]{SW11}D. Sieg, S. Wegner, \emph{Maximal exact structures on additive categories,} Math. Nachr. 284, No. 16, 2093-2100 (2011).
\bibitem[Sch99]{Sch}J-P. Schneiders, \emph{Quasi-abelian categories and sheaves}, M\'em. Soc. Math.Fr. (N.S.) (1999), no.76, vi+134.

\bibitem[T19]{T19} A. Tattar, \emph{Torsion pairs and quasi-abelian categories}, Algebras and Representation Theory (2020), 1--25
 

	
\bibitem[VW20]{VW}	Y. Volkov, S. Witherspoon, {\em Graded Lie structure on cohomology of some exact monoidal categories},
	arXiv:2004.06225.


\end{thebibliography}
\end{document}